\documentclass[final,11pt]{article}

\usepackage[english]{babel}
\usepackage[utf8]{inputenc}

\usepackage{mathpazo}

\usepackage{xfrac}
\usepackage{braket}
\usepackage{amsmath}
\usepackage{amstext}
\usepackage{amssymb}
\usepackage{amsthm}
\usepackage{mathtools}
\usepackage{stmaryrd}
\usepackage{mathabx}
\usepackage{cancel}

\usepackage{fullpage}

\usepackage[final]{microtype}

\usepackage[colorlinks]{hyperref}

\usepackage[notref,notcite,color]{showkeys}

\usepackage{ifdraft}
\usepackage{wrapfig}
\usepackage{subfig}
\ifdraft{\usepackage{showkeys}}{}
\usepackage[Sonny]{fncychap}
\usepackage{suffix}
\usepackage{booktabs}

\usepackage{todonotes}
\usepackage{setspace}
\swapnumbers

\numberwithin{equation}{section}

\theoremstyle{plain}
\newtheorem{theo}{Theorem}[section]
\newtheorem{pro}[theo]{Proposition}
\newtheorem{cor}[theo]{Corollary}
\newtheorem{lem}[theo]{Lemma}

\newtheorem{rem}[theo]{Remark}

\newtheorem{defi}[theo]{Definition}

\DeclareMathOperator{\Ric}{Ric}

\DeclareMathOperator{\Div}{div}

\newcommand{\ooint}[2]{\mathopen (#1,#2\mathclose )}
\newcommand{\ocint}[2]{\mathopen (#1,#2\mathclose ]}
\newcommand{\coint}[2]{\mathopen [#1,#2\mathclose )}
\newcommand{\ccint}[2]{\mathopen [#1,#2\mathclose ]}
\WithSuffix\newcommand\ooint*[2]{\left(#1,\,#2\right)}
\WithSuffix\newcommand\ocint*[2]{\left(#1,\,#2\right]}
\WithSuffix\newcommand\coint*[2]{\left[#1,\,#2\right)}
\WithSuffix\newcommand\ccint*[2]{\left[#1,\,#2\right]}

\newcommand{\C}{\mathcal{C}}

\newcommand{\R}{\mathbb{R}}

\let\d\relax

\newcommand{\d}{\, d}
\newcommand{\de}{\partial}

\newcommand{\scal}[2]{\braket{#1 | #2}}
\newcommand{\Scal}[2]{\Braket{#1 | #2}}
\DeclarePairedDelimiter{\abs}{\lvert}{\rvert}

\DeclareMathOperator{\Ca}{Cap}

\renewcommand{\phi}{\varphi}
\renewcommand{\epsilon}{\varepsilon}
\let\oldchi\chi
\renewcommand{\chi}{\raisebox{.2ex}{$\oldchi$}}
\renewcommand{\leq}{\leqslant}
\renewcommand{\geq}{\geqslant}

\usepackage[affil-it]{authblk}

\makeatletter
\def\@maketitle{%
	\newpage
	\null
	\vskip 2em%
	\begin{center}%
		\let \footnote \thanks
		{\Large\bfseries \@title \par}%
		\vskip 1.5em%
		{\normalsize
			\lineskip .5em%
			\begin{tabular}[t]{c}%
				\@author
			\end{tabular}\par}%
	\end{center}%
	\par
	\vskip 1.5em}
\makeatother

\title{Some Sphere Theorems in Linear Potential Theory}

\author[1,2]{Stefano Borghini}
\author[1]{Giovanni Mascellani}
\author[2]{Lorenzo Mazzieri}
\affil[1]{\small Scuola Normale Superiore di Pisa,
	piazza dei Cavalieri 7, 56126 Pisa (PI), Italy}
\affil[2]{\small Universit\`a degli Studi di Trento,
via Sommarive 14, 38123 Povo (TN), Italy}
\affil[ ]{\vskip 0em\rm stefano.borghini@unitn.it , 
giovanni.mascellani@sns.it , 
lorenzo.mazzieri@unitn.it}

\date{}

\makeatletter
\newcommand{\customlabel}[2]{%
  \protected@write\@auxout{}{\string\newlabel{#1}{{#2}{\thepage}{#2}{#1}{}}}%
  \hypertarget{#1}{}
}
\makeatother

\renewcommand\appendix{\par
\setcounter{subsection}{0}%
\setcounter{table}{0}
\setcounter{figure}{0}
\gdef\thetable{\Alph{table}}
\gdef\thefigure{\Alph{figure}}
\gdef\thesection{A}
}

\begin{document}

\maketitle

\abstract{\noindent In this paper we analyze the capacitary potential due to a charged body in order to deduce sharp analytic and geometric inequalities, whose equality cases are saturated by domains with spherical symmetry. In particular, for a regular bounded domain $\Omega \subset \R^n$, $n\geq 3$, we prove that if the mean curvature $H$ of the boundary obeys the condition 
\begin{equation*}
- \, \bigg[ \frac{1}{\Ca(\Omega)} \bigg]^{\frac{1}{n-2}} \!\! \leq \,\, \frac{H}{n-1}   \,\, \leq \,\, \bigg[ \frac{1}{\Ca(\Omega)} \bigg]^{\frac{1}{n-2}} \, ,
\end{equation*}	
then $\Omega$ is a round ball.}
\\
\\
\noindent\textsc{MSC (2010):
	35N25,
	\!31B15,
	\!35B06,
	\!53C21.
}

\noindent{\underline{Keywords}: capacity, electrostatic potential, overdetermined boundary value problems.}

\section{Introduction and statement of the main result}

In this work we are interested in studying the analytic and geometric properties of a solution $(\Omega, u)$ to the following problem
	\begin{equation}
	\label{eq:system}
	\def\arraystretch{1.3}
	\left\{ \begin{array}{r@{}c@{}ll}
		\Delta u \,& \,=\, &\, 0 & \text{in } \R^n \setminus \overline\Omega\, , \\
		u \,&\,=\,&\, 1 & \text{on } \de\Omega \, , \\
		u(x) \,&\,\to\,&\, 0 & \text{as } \abs{x} \to +\infty \,.
		\end{array}\right.
	\end{equation}  	
In classical potential theory, the bounded domain $\Omega$ is interpreted as a uniformly charged body and the function $u$ is the electrostatic (or capacitary) potential due to $\Omega$. 
The basic properties of solutions to~\eqref{eq:system}, such as existence, uniqueness, regularity, validity of maximum and minimum principles as well as asymptotic estimates, are of course very well known and for a systematic treatment, we refer the reader to~\cite{kellogg}. However, for the statements which are more relevant to our work, we decided to include the proofs in Appendix~\ref{ap:asympt}, where other references are also given.
%
%
%
%
%
In order to introduce our main results, we specify our assumptions on the regularity of both the domain $\Omega$ and the function $u$ with the following definition.
\begin{defi}
	 \label{ass:main}
A pair $(\Omega,u)$ is said to be a {\em regular solution} to problem~\eqref{eq:system} if
\begin{itemize}
	\item $\Omega  \subset \R^n$, $n \geq 3$, is a bounded open domain with smooth boundary. Without loss of generality, we also assume that $\Omega$ contains the origin and $\R^n\setminus\Omega$ is connected.
	\item The function $u \colon
	\R^n \setminus \Omega \to \ocint{0}{1}$ is a solution to problem~\eqref{eq:system}, it is $\mathcal{C}^{\infty}$ in $\R^n \setminus \Omega$ and analytic in $\R^n \setminus \overline{\Omega}$.
\end{itemize}
\end{defi}
\noindent
Before proceeding, we point out that the assumption about the connectedness of $\R^n\setminus\Omega$ is not restrictive. In fact, since $\Omega$ is bounded, the set $\R^n\setminus\Omega$ has one and only one unbounded connected component. The other ones are then forced to be compact so that $u$ has to be constant on them. In other words, these connected components are completely irrelevant to the analysis of problem~\eqref{eq:system} and the above definition allows us to simply ignore them. 
%
%
%
%
%
%
%

In order to fix some notation as well as some convention, we recall that the  \emph{electrostatic capacity} associated to the set $\Omega$ is defined by
\begin{equation}
  \label{eq:def-omega}
  \Ca(\Omega) = \inf \Set{\frac{1}{(n-2)\abs{S^{n-1}}}\int_{\R^n}\abs{D w}^2\d v | w\in\C^\infty_c(\R^n) ,\ w\equiv 1\ \mbox{in } \Omega  } \, ,
\end{equation}
where $\abs{S^{n-1}}$ is the area of the unit sphere $S^{n-1}$ in
$\R^n$ and $d \sigma$ is the hypersurface measure in $\R^n$. We refer the reader to Appendix~\ref{ap:asympt} for some comments about this definition, here we just observe that the capacity of $\Omega$ can be represented in terms of the {\em capacitary potential} $u$ as
\begin{equation}
\label{eq:capacity}
\Ca(\Omega) =  \frac{1}{(n-2) \abs{S^{n-1}}} \int_{\de\Omega} \abs{Du} \d \sigma \, .
\end{equation}
We recall that if $\Omega$ is a round ball centered at the origin, then the capacitary potential
$u$ is radially symmetric and it is given by the explicit formula $u(x) = \Ca(\Omega) \abs{x}^{2-n}$. For more general domains, $u$ satisfies the asymptotic expansions described in Proposition~\ref{pro:est-u}.

To introduce our main result, we recall that
in the recent work~\cite{agostiniani-mazzieri} the authors define for $p \geq 0$ the quantities
\begin{equation}
\label{eq:U_p}
(0,1] \, \ni  \, t \,\, \longmapsto \,\, U_p(t) \, 
\,= 
\,\, \left[\frac{{\rm Cap} (\Omega)}{t} \right]^{\frac{(p-1)(n-1)}{(n-2)}}
\!\!\!\!\!\!\! \int_{ \{ u = t \}} \!\!\!  |D u|^p \, d \sigma 
\end{equation}
and prove that for every $p \geq 2- 1/(n-1)$ these functions are differentiable and monotonically nondecreasing, namely $U_p' \geq 0$. Moreover, the monotonicity is strict unless both the function $u$ and the domain $\Omega$ are rotationally symmetric (see~\cite[Theorem~1.1]{agostiniani-mazzieri}). On the other hand, by the asymptotic expansion of $u$ and $|D u|$ at infinity, one can easily compute the limit
\begin{equation}
\label{eq:lim_up}
\lim_{t\to 0^{+}} \, U_p(t) \,\, = \,\, \left[{{\rm Cap} (\Omega)}\right]^{p} \!(n-2)^p\, \abs{S^{n-1}} \, .\phantom{\qquad}
\end{equation}
Combining this fact with the monotonicity, the authors are able to prove the following sharp inequality \begin{equation*}
\phantom{\quad \quad}
|S^{n-1}|^{\!\frac1p} \, \left[ {{\rm Cap}(\Omega)}\right]^{1-\frac{(p-1)(n-1)}{p(n-2)}} \!\!\! \leq \,\,  \,  \left\|  \frac{ \, D  u \, }{n-2} \right\|_{L^p(\partial\Omega)}  ,
\end{equation*}
for every $p \geq 2- 1/(n-1)$. 
An interesting feature of the above inequality is that it comes together with a rigidity statement. In fact, as soon as the equality is fulfilled, everything is rotationally symmetric. Finally, letting $p \to +\infty$, one deduces that 
\begin{equation*}
\phantom{\quad \quad}
 \left[ {{\rm Cap}(\Omega)}\right]^{-\frac{1}{n-2}}  \leq \,\,  \, \max_{\de\Omega}\left|\frac{Du}{n-2}\right| \,.
\end{equation*}
However, {\em a priori} it is not clear if the rigidity statement holds true for this inequality. 
Our main result answers in the affermative to this question.


\begin{theo}
\label{theo:main-thm2}
Let $(\Omega,u)$ be a regular solution to problem~\eqref{eq:system} in the sense of Definition~\ref{ass:main}. Then the following inequality holds
	\begin{equation}
	\label{eq:main-thm2}
	\frac{1}{\Ca(\Omega)} \,\leq\,\left({\max_{\de\Omega}\left|\frac{Du}{n-2}\right|}\right)^{n-2}\,.
	\end{equation}
Moreover, the equality is fulfilled if and only if $\Omega$ is a ball and $u$ is rotationally symmetric.
\end{theo}
%
%
%
%

\begin{rem}
\label{rem:PP}
In dimension $n=3$, a proof of Theorem~\ref{theo:main-thm2} has already been provided by Payne and Philippin in~\cite[Formula~(3.24)]{Pay_Phi}. This $3$-dimensional proof is based on~\cite[Theorem in Section~2]{Pay_Phi} and the well known Poincar\'{e}-Faber-Szeg\"o inequality
\begin{equation}
\label{eq:PFS}
|\Omega|\,\leq\,\frac{4\pi}{3}[\Ca(\Omega)]^3\,,
\end{equation}
which holds on every bounded domain of $\R^3$, see~\cite[Section~1.12]{Pol_Sze}.
The proof in~\cite{Pay_Phi} can be extended to any dimension $n\geq 3$ using a generalized version of~\eqref{eq:PFS} (see Theorem~\ref{theo:PFS} in the appendix)
\begin{equation}
\label{eq:PFS_2}
\left(\frac{|\Omega|}{|B^n|}\right)^{\!\!\frac{n-2}{n}}\!\!\!\leq\,\Ca(\Omega)\,,
\end{equation}
where $B^n\subset\R^n$ is the $n$-dimensional unit ball. This inequality has also been extended to conformally flat manifolds in~\cite[Theorem~1]{Fre_Sch}. 
However, this is not enough to generalize the approach of Payne and Philippin beyond the Euclidean setting, since the computations in~\cite{Pay_Phi} make use of integral identities that exploits the peculiar structure of $\R^n$. In contrast, we emphasize that the proof proposed in the present work is self contained, provides a clear geometric interpretation of the result (see Theorem~\ref{theo:main-thm2-f}), and does not depend heavily on the structure of $\R^n$. In particular, it seems that the method used in this work should be generalizable to the exterior problem on a general manifold, requiring only the nonparabolicity of the ends (so that a solution of the analogue of system~\eqref{eq:system} exists) and suitable asymptotic estimates of the potential $u$. This will be the object of further investigations.
\end{rem}

\noindent 
The proof of a slightly extended version of Theorem~\ref{theo:main-thm2}  (see Theorem~\ref{theo:main-thm-ext2})
will be presented in full details in Section~\ref{sec:proof-part2}. However, before proceeding, it is important to remark that  the new rigidity statement is obtained by exploiting the peculiar geometric features of the new method ({\em spherical ansatz}), in contrast with the method employed in~\cite{agostiniani-mazzieri} ({\em cylindrical ansatz}), for which this result was clearly out of reach. For further comments about the way the rigidity is proven and for a more detailed comparison between the two methods, we refer the reader to Section~\ref{sec:strategy}.

\subsection{Analytic consequences - Sphere Theorem I}

As a first implication of the above result we obtain the following sphere theorem, which roughly speaking says that if the strength of the electric field $|Du|$ measured on the surface of the uniformly charged body $\Omega$ is bounded above by the inverse of the so called surface radius $(|\partial \Omega| / |S^{n-1}|)^{1/(n-1)}$, then the charged body must be rotationally symmetric.
\begin{cor}[Sphere Theorem I]
\label{teo:sph_uno}
Let $(\Omega,u)$ be a regular solution to problem~\eqref{eq:system} in the sense of Definition~\ref{ass:main}, and suppose that
\begin{equation}
\label{eq:ov_uno}
\left|\frac{Du}{n-2}\right| \,\, \leq  \,\,\bigg( \frac{|S^{n-1}|}{|\partial \Omega|} \bigg)^{\!\!\frac{1}{n-1}} \, \quad \hbox{on $\partial \Omega$}\, .
\end{equation}	
Then $\partial \Omega$ is a sphere and $u$ is rotationally symmetric.
\end{cor}
\begin{proof}
Dividing both sides of~\eqref{eq:ov_uno} by $|S^{n-1}|$ and integrating over $\partial \Omega$ gives
\begin{equation*}
\Ca(\Omega) \,\, = \,\, \frac{1}{(n-2) \abs{S^{n-1}}} \int_{\de\Omega} \abs{Du} \d \sigma \,\, \leq \,\, \bigg( \frac{|S^{n-1}|}{|\partial \Omega|} \bigg)^{\! -\frac{n-2}{n-1}} \, .
\end{equation*}
Combining this inequality with~\eqref{eq:main-thm2} and using again the hypothesis~\eqref{eq:ov_uno}, we obtain the following chain of inequalities
\begin{equation*}
\bigg( \frac{|S^{n-1}|}{|\partial \Omega|} \bigg)^{\!\frac{n-2}{n-1}} \,\, \leq \,\, \frac{1}{\Ca(\Omega)} \,\, \leq \,\, \left({\max_{\de\Omega}\left|\frac{Du}{n-2}\right|}\right)^{n-2} \,\, \leq \,\, \bigg( \frac{|S^{n-1}|}{|\partial \Omega|} \bigg)^{\!\frac{n-2}{n-1}} \, .
\end{equation*} 
In particular, the second inequality holds with the equality sign and the conclusion follows by the rigidity statement of Theorem~\ref{theo:main-thm2}. 
\end{proof}
\noindent The above statement should be compared with~\cite[Theorem~2]{Reichel}, where a similar sphere theorem is proven. In terms of our notation, Reichel's Theorem reads as follows:
\begin{theo}[Reichel~\cite{Reichel}]
\label{teo:reichel}
Let $(\Omega,u)$ be a regular solution to problem~\eqref{eq:system} in the sense of Definition~\ref{ass:main} and suppose that 
\begin{equation}
\label{eq:ov_reichel}
|{Du}| \,\, \equiv  \,\, c \, \quad \hbox{on $\partial \Omega$}\, ,
\end{equation}	
for some positive constant $c>0$. Then $\partial \Omega$ is a sphere and $u$ is rotationally symmetric.
\end{theo}

\begin{rem}
Reichel's Theorem has been generalized in a number of directions. A different proof of Theorem~\ref{teo:reichel} has appeared in~\cite{Bia_Cir_Sal}, and similar results have been proved for other classes of elliptic inequalities~\cite{Reichel3}, for annulus-type domains~\cite{Reichel2} and for the $p$-laplacian~\cite{Gar_Sar,Bia_Cir_2}. The case of a disconnected $\Omega$ --with possibly different Dirichlet and Neumann condition on the different boundary components-- is analyzed in~\cite{Sirakov}. 
It seems interesting to ask if our Corollary~\ref{teo:sph_uno} admits similar generalizations.
We just point out that our main result extends without effort to the case in which $\Omega$ is not connected, but assuming that we have the same Dirichlet condition $u=1$ on the whole $\de\Omega$.
\end{rem}
\noindent To understand the relationship between Reichel's overdetermining assumption~\eqref{eq:ov_reichel} and our assumption~\eqref{eq:ov_uno} we observe that in virtue of our Theorem~\ref{theo:main-thm2}, the constant $c$ that appears in Reichel's hypothesis is no longer allowed to be any positive constant, but it must obey the {\em a priori} constraint 
\begin{equation*}
(n-2) \left[ \frac{1}{\Ca(\Omega) }\right]^{\frac{1}{n-2}} \!\! \leq \,\, c \,.
\end{equation*}
Using again Reichel's overdetermining condition~\eqref{eq:ov_reichel}, one can express the capacity of $\Omega$ as 
\begin{equation*}
\Ca(\Omega)  \,\, = \,\, \frac{c}{n-2} \, \bigg(\frac{|\partial \Omega|}{|S^{n-1}|} \bigg)
\end{equation*}
so that, up to some algebra, the above {\em a priori} constraint can be written as
\begin{equation*}
\bigg( \frac{|S^{n-1}|}{|\partial \Omega|} \bigg)^{\!\frac{1}{n-1}} \!\! \leq \,\, \frac{c}{n-2} \,\, \equiv \,\,
\left| \frac{Du}{n-2} \right|  \, \quad \hbox{on $\partial \Omega$}\, .
\end{equation*}
Having this in mind, we can rephrase our Corollary~\ref{teo:sph_uno} as follows.
\begin{cor}[Sphere Theorem I - Revisited]
If $(\Omega, u)$ is a regular solution to problem~\eqref{eq:system} in the sense of Definition~\ref{ass:main}, and we suppose that its gradient $|Du|$ at $\partial \Omega$ is bounded from above by the least admissible constant Neumann data in Reichel's overdetermining prescription~\eqref{eq:ov_reichel}, then $\partial \Omega$ is a sphere and $u$ is rotationally symmetric.
\end{cor}


%
%
%
%
%
%
%
%
%
%
%
%
%

\subsection{Geometric consequences - Sphere Theorem II}

In order to state some purely geometric consequence of Theorem~\ref{theo:main-thm2}, we recall from~\cite[formula (2.6)]{agostiniani-mazzieri} the inequality
\begin{equation}
\label{eq:AM-1}
\max_{\de\Omega}\left|\frac{D u}{n-2}\right|\,\leq\,\max_{\de\Omega}\left|\frac{H}{n-1}\right|\,,
\end{equation}
which holds for regular solutions of~\eqref{eq:system}. Here $H$ is the mean curvature of the hypersurface $\de\Omega$ with respect to the normal pointing towards the interior of $\Omega$
and the standard flat metric of $\R^n$. We precise that, in our convention, the mean curvature $H$ is defined as the sum of the principal curvatures. Combining~\eqref{eq:AM-1} with~\eqref{eq:main-thm2} we obtain a geometric inequality (see~\eqref{eq:AM-3} below) involving only the capacity of $\Omega$ and the mean curvature of its boundary, whose equality case is characterized in terms of the rotational symmetry of $\Omega$. This latter rigidity statement answers to a question raised in~\cite{agostiniani-mazzieri} as well as~\cite{Ago_Maz_RendicontiLincei}. 
\begin{theo}
\label{teo:main_geo}
Let $\Omega \subset \R^n$ be a bounded domain with smooth boundary and $n \geq 3$. Then the following inequality holds
	\begin{equation}
	\label{eq:AM-3}
	\frac{1}{\Ca(\Omega)}  \,\leq\,    \left({\max_{\de\Omega}\left|\frac{H}{n-1}\right|}\right)^{\!n-2}     \,.
	\end{equation}
Moreover, the equality is fulfilled if and only if $\Omega$ is a round ball.	
\end{theo}
\noindent In other words, Theorem~\ref{teo:main_geo} tells us that among regular bounded domains for which the maximum of the mean curvature of the boundary is fixed, the capacity is minimized by spheres. 
As an immediate corollary of Theorem~\ref{teo:main_geo}, we obtain the following sphere theorem, that should be compared with the classical Alexandrov Theorem for compact connected constant mean curvature hypersurfaces in Eucidean space. Here we draw the same conclusion under a pointwise pinching assumption for the mean curvature. 
\begin{cor}[Sphere Theorem II]
\label{teo:sph_due}
Let $\Omega \subset \R^n$ be a bounded domain with smooth boundary and $n \geq 3$. Suppose that
\begin{equation*}
- \, \bigg[ \frac{1}{\Ca(\Omega)} \bigg]^{\frac{1}{n-2}} \!\! \leq \,\, \frac{H}{n-1}   \,\, \leq \,\, \bigg[ \frac{1}{\Ca(\Omega)} \bigg]^{\frac{1}{n-2}} \, .
\end{equation*}	
Then $\partial \Omega$ is a sphere.
\end{cor}

\begin{rem}
This result was already provided in~\cite[Formula~(3.29)]{Pay_Phi} in dimension $n=3$.
In the case where $\Omega$ is convex and $n=3$, an easier proof is available, see~\cite{Bandle}.
\end{rem}

\noindent We observe that the relation between our Sphere Theorem II and the celebrated Alexandrov Theorem is the perfect analogous of the one described in the previous subsection between our Sphere Theorem I and Reichel's result. In order to see this, let us rephrase the Alexandrov Theorem in our setting.
\begin{theo}[Alexandrov~\cite{Aleksandrov}]
\label{teo:alex}
Let $\Omega \subset \R^n$ be a bounded domain with smooth boundary and $n \geq 3$. Suppose that
\begin{equation}
\label{eq:ov_alex}
H \,\, \equiv  \,\, c \, \quad \hbox{on $\partial \Omega$}\, ,
\end{equation}	
for some constant $c >0$. Then $\partial \Omega$ is a sphere.
\end{theo}

\begin{rem}
As for Reichel's Theorem, also Alexandrov's Theorem has been generalized in several ways. A similar result for the $k$-mean curvature have been proved in~\cite{Ros,Mon_Ros}, and a quantitative version of Alexandrov's Theorem appeared in~\cite{Cir_Vez}. Also, a characterization of the constant mean curvature hypersurfaces have been found in other space-forms~\cite{Alexandrov2} and in warped product manifolds (see~\cite{Brendle,Bre_Eic,Montiel}).
Corollary~\ref{teo:sph_due} follows yet another direction, and  can be interpreted as a rigidity result under a pinching condition on the mean curvature. Since the technique used in this work does not seem to depend deeply on the structure of the Euclidean space, it is natural to ask if also our result extends to other manifolds.
\end{rem}
\noindent In the above statement we drop any attempt of writing the optimal regularity assumptions. Moreover, we observe that by the compactness of $\partial \Omega$ the assumption on the sign of the constant $c$ is not restrictive. On the other hand, in virtue of Theorem~\ref{teo:main_geo}, the constant $c$ in~\eqref{eq:ov_alex} is {\em a priori} forced to obey the constraint
\begin{equation*}
\bigg[ \frac{1}{\Ca(\Omega)} \bigg]^{\frac{1}{n-2}}
\!\! \leq \,\, \frac{c}{n-1} \,\, \equiv \,\,
\left| \frac{H}{n-1} \right|  \, \quad \hbox{on $\partial \Omega$}\, .
\end{equation*}
Hence, it is clear that our Corollary~\ref{teo:sph_due} can be rephrased as follows.
\begin{cor}[Sphere Theorem II - Revisited]
If a bounded domain $\Omega \subset \R^n$  with smooth boundary and $n \geq 3$ is such that the absolute value of the mean curvature of its boundary is bounded from above by the least admissible constant in~\eqref{eq:ov_alex}, then $\partial \Omega$ is a sphere.
\end{cor}

\noindent 
We have already noticed in Remark~\ref{rem:PP} that the proof of the above results does not depend on the Poincar\'{e}-Faber-Szeg\"o inequality~\eqref{eq:PFS_2}. However,~\eqref{eq:PFS_2} can be combined with the above results in order to deduce another Sphere Theorem for starshaped domains. Raising to the square both sides of~\eqref{eq:AM-1}, multiplying them by $\scal{x}{\nu}$, where $x$ is the position vector and $\nu$ is the exterior unit-normal to $\de\Omega$, and then integrating on $\de\Omega$, we obtain
$$
\int_{\de \Omega}\left|\frac{D u}{n-2}\right|^2\scal{x}{\nu}\d\sigma\,\leq\,\max_{\de\Omega}\left|\frac{H}{n-1}\right|^2\int_{\de\Omega}\scal{x}{\nu}\d\sigma
\,=\,n|\Omega|\max_{\de\Omega}\left|\frac{H}{n-1}\right|^2
$$
Using formula~\eqref{eq:capacity_alter_2} in the appendix, the term on the left hand side is equal to $|S^{n-1}|\Ca(\Omega)=n|B^n|\Ca(\Omega)$. Using the Poincar\'{e}-Faber-Szeg\"o inequality~\eqref{eq:PFS_2} to estimate $\Ca(\Omega)$, with some easy algebra we obtain the following result.
\begin{cor}[Sphere Theorem for starshaped domains]
Let $\Omega \subset \R^n$ be a bounded starshaped domain with smooth boundary and $n \geq 3$. Suppose that
	\begin{equation}
	\label{eq:conj}
	-\left(\frac{|B^n|}{|\Omega|}\right)^{\!\frac{1}{n}}\leq\frac{H}{n-1}\leq\left(\frac{|B^n|}{|\Omega|}\right)^{\!\frac{1}{n}}  
	\end{equation}
	where $B^n$ is the $n$-dimensional unit ball. 
Then $\de\Omega$ is a sphere.
\end{cor}

\noindent Recalling~\eqref{eq:PFS_2}, we see that this corollary is a stronger version of Theorem~\ref{teo:main_geo}.  Unfortunately, our proof only works for starshaped domain. It would be interesting to study if this hypotesis can be relaxed.

%
%

\section{Geometric methods and strategy of the proof}
\label{sec:strategy}

In this section we outline the strategy of the proof of the previously mentioned results and we set up some notations. This gives us the opportunity to describe in details the analogies and differences between our method and the one introduced in~\cite{agostiniani-mazzieri-14, agostiniani-mazzieri}, hoping that this comparison will be useful for further applications. 

\subsection{Spherical ansatz vs cylindrical ansatz}

%
%
%
%

The method presented in~\cite{agostiniani-mazzieri-14, agostiniani-mazzieri} relies on a geometric construction called {\em cylindrical ansatz}. One starts from the simple observation that the Euclidean metric $g_{\R^n}$ can be written as
$$
g_{\R^n}\,=\,d\abs{x}\otimes d\abs{x} +\abs{x}^{2}g_{S^{n-1}}\,=\,\abs{x}^{2}\,\Big(\abs{x}^{-2}d\abs{x}\otimes d\abs{x} +g_{S^{n-1}}\Big)
\,=\,\abs{x}^{2}\,g_{\rm cyl}\,=\,u_0^{-\frac{2}{n-2}}\,g_{\rm cyl}\,,
$$ 
where $g_{\rm cyl}$ is the cylindrical metric on $\R^n\setminus\{0\}$ and we have denoted by
$u_0(x) = \abs{x}^{2-n}$ the rotationally symmetric solution to
problem~\eqref{eq:system} when $\Omega$ is
a ball centered at the origin with radius $1$. With this in mind, given a generic solution $(\Omega,u)$ to problem~\eqref{eq:system}, one is lead to consider the ansatz metric
$$
g\,=\,u^{\frac{2}{n-2}}\, g_{\R^n}\,,
$$
and study under which conditions $g$ is actually cylindrical. To proceed, one sets $\varphi=\log(u)$ and notices that, when $(\Omega,u)$ is a rotationally symmetric solution, then $\nabla\varphi$ represents the splitting direction of the cylinder. For this reason, the strategy of the proof in~\cite{agostiniani-mazzieri-14, agostiniani-mazzieri} is to find suitable conditions under which some {\em splitting principle} applies.  When this is the case, one recovers the rotational symmetry of the solution $(\Omega,u)$.

In this work, we introduce a different geometric construction, called {\em spherical ansatz}, and we aim to investigate its consequences along the same lines as the ones just described for the {\em cylindrical ansatz}. The main difference is that this time the rigidity statements will be deduced from a suitable {\em sphere theorem} instead of a {\em splitting principle}. To be more concrete, we indicate with $S^n = \{ (z^1, \ldots, z^{n+1}) \in \R^{n+1} \,\, :   \,\, |z|^2 = 1  \, \}$ the unit sphere sitting in the Euclidean space $\R^{n+1}$ and with $N$ its north pole $(0, \dots, 0,
1)$. We let then $\pi$ be the stereographic projection that maps conformally
$S^n \setminus \{N\}$ onto the $n$-dimensional plane $ \{ (z^1, \ldots, z^{n+1}) \in \R^{n+1} \,\, :   \,\, z^{n+1} = 0  \, \}$. We set $z' = (z^1, \ldots, z^{n})$ and recall the precise
formul\ae:
\begin{align*}
  \pi \colon S^n \setminus \{N\} & \longrightarrow \R^n & \pi^{-1} \colon \R^n & \longrightarrow S^n \setminus \{N\} \\
  (z', z^{n+1}) & \longmapsto \frac{z'}{1-z^{n+1}} & x & \longmapsto \left(\frac{2x}{\abs{x}^2+1}, \frac{\abs{x}^2-1}{\abs{x}^2+1}\right) .
\end{align*}
Pulling back the standard round metric $g_{S^n}$ (i.e., the one induced by the Euclidean metric $g_{\R^{n+1}}$
on the unit sphere $S^n$) one obtains on $\R^n$ the metric $g_\text{sph} = (\pi^{-1})^* g_{S^n}$, so that 
$(\R^n, g_\text{sph})$ is isometric to a sphere, except for a missing point. The
explicit form of such metric is
\[ g_\text{sph} = (\pi^{-1})^* g_{S^n} = \left( \frac{2}{1+\abs{x}^2} \right)^{\!\!2}\! g_{\R^n} = \left( \frac{2}{1+u_0^{-\frac{2}{n-2}}} \right)^{\!2} g_{\R^n} , \]
where we recall that $u_0(x) = \abs{x}^{2-n}$. Therefore, starting from a generic regular solution $(\Omega, u)$ to problem~\eqref{eq:system}, we are lead to consider the \emph{ansatz} metric
\begin{equation}
  \label{eq:def-g}
  g = \left( \frac{2}{1+u^{-\frac{2}{n-2}}} \right)^{\!2} g_{\R^n} \,. 
\end{equation}
In order to obtain rigidity statements, we need to investigate under which conditions the Riemannian manifold $(\R^n \setminus \overline{\Omega})$ is isometric to an hemisphere.
This will imply in turns the rotational symmetry of both $\Omega$ and $u$. As anticipated, sphere theorems will play here the same role as the splitting principles in the {\em cylindrical ansatz} situation.

\begin{rem}
\label{rem:different-radiai}
It should be noticed that if the domain $\Omega$ is actually a ball of radius $\rho$, then the solution $u$ is rotationally symmetric 
and it is explicitly given by $u = \rho^{n-2} |x|^{2-n}$. On the other hand, by Proposition~\ref{pro:est-u}, one has that
\begin{equation*}
u \, = \, \Ca(\Omega)\abs{x}^{2-n} \, ,
\end{equation*}
so that the radius $\rho$ can be computed in terms of $\Omega$ according to the formula
\begin{equation*}
\rho \,\, = \,\, {[\Ca(\Omega)]}^{\frac{1}{n-2}} \,.
\end{equation*}
In this case, formula~\eqref{eq:def-g} rewrites as
\begin{align*}
 g&\, =\, \left( \frac{2}{1+{[\Ca(\Omega)]}^{-\frac{2}{n-2}}\abs{x}^2} \right)^{\!2}   \cdot \,dx^1\!\otimes dx^1 \!+ \ldots + dx^n\!\otimes dx^n \, = \,  \\
 &\,=\,{[\Ca(\Omega)]}^{\frac{2}{n-2}}\left( \frac{2}{1+\abs{\xi}^2} \right)^{\!2} \cdot \,d\xi^1\!\otimes d\xi^1 \!+ \ldots + d\xi^n\!\otimes d\xi^n \,=\,{[\Ca(\Omega)]}^{\frac{2}{n-2}} \, g_{\rm sph}\,,
\end{align*}
where in the second equality we have used the change of coordinates $\xi={[\Ca(\Omega)]}^{\frac{1}{n-2}}x$, and in the last equality we have used the definition of $g_{\rm sph}$.
This means that, in general, in what concerns the rigidity statements, we expect the manifold $(\R^n\setminus\overline{\Omega},g)$ to be isometric to a round hemisphere of radius $\rho={[\Ca(\Omega)]}^{{1}/({n-2})}$.
%
%
\end{rem}
\noindent In the case under consideration, it seems reasonable to proceed in the direction of the classical Obata rigidity theorem~\cite[Theorem~A]{obata}, according to which a necessary and sufficient condition
for a closed manifold to be isometric to a unit-radius sphere is to admit a function $f$
satisfying the condition $\nabla^2 f = -fg$. On the standard sphere
$S^n \subset \R^{n+1}$ such condition is verified by the restriction
to the sphere itself of any coordinate function of $\R^{n+1}$. So in
our case it is natural to supplement the ansatz metric $g$ with a
function $f$ that is the candidate to trigger the rigidity theorem. In
order to find such a function, we turn ourselves again to the model
case and consider the function
\begin{equation} 
\label{eq:f0}
f_0(x) \, = \,  \big( (\pi^{-1})^* z^{n+1} \big) (x) \, = \,  z^{n+1} \circ \pi^{-1}(x) \, = \,  \frac{\abs{x}^2\!-1}{\abs{x}^2\!+1} \,  = \, \frac{u_0^{-\frac{2}{n-2}} \!-1}{u_0^{-\frac{2}{n-2}}\!+1} \, .
\end{equation}
Given a generic solution $(\Omega, u)$ to~\eqref{eq:system}, it will then sound reasonable to set
\begin{equation}
  \label{eq:def-f}
  f \,\, = \,\,  \frac{u^{-\frac{2}{n-2}}\!-1}{u^{-\frac{2}{n-2}}\!+1} \,\, = \,\, -\tanh \left( \frac{\log u}{n-2} \right)
\end{equation}
and study the equations it solves.
In light of the above formul\ae, and depending on the context, it will be convenient to consider $f$ either as a function of
$u$ or as a function of the point $x$, hopefully without confusing the reader. Similarly we will possibly consider $u$ as a function of either $f$ or $x$, depending on the con\-text.
In order to understand the domain of definition as well as the range of $f$, we observe that by the classical strong maximum principle (see~\cite{evans-pde} or Appendix~\ref{ap:asympt}), when the point $x$ varies in $\R^n \setminus \overline
\Omega$, the function $u(x)$ takes values in $\ooint{0}{1}$. On the other hand, it is clear by formula~\eqref{eq:def-f} that $f$ is strictly decreasing as a function of $u$. Moreover it is
equal to $0$ on $\de \Omega$, tends to $1$ as $\abs{x} \to \infty$ and takes values in $\ooint{0}{1}$ when its argument is ranging in $\R^n \setminus \overline \Omega$. A consequence of this fact is that the level sets of both $u$ and $f$
are compact, which in turn implies that the set of regular values is
open in the codomain. By Sard's lemma it is also a set of full
measure, so in particular it is dense. Also, by construction, $f$ and
$u$ share the same level sets (although with opposite choices of
sublevel sets) and each of them is regular for either both or none of
$u$ and $f$.


As we will show in Section~\ref{sec:conformal-form}, the relation between the ansatz metric $g$ and the Euclidean metric $g_{\R^n}$ can be expressed in terms of $f$ as
$$
g=(1-f)^2g_{\R^n} \, ,
$$
so that the original problem~\eqref{eq:system} can be rewritten in terms of $f$ and $g$ as

\begin{equation}
\label{eq:geom-system}
\def\arraystretch{1.8}
\left\{ \begin{array}{r@{}c@{}ll}
\Delta_g f \,&\,=\, &\,-n \left(\frac{\abs{\nabla f}_g^2}{1-f^2} \right) \, f & \text{in } \R^n \setminus \overline\Omega \, ,
\\
\Ric \,&\, = \,&\, \frac{n-2}{1-f} \,\, \nabla^2 f \,\, + \,\, \frac{n-1-f}{1-f} \, \left(\frac{\abs{\nabla f}_g^2 }{1-f^2}\right) \,g & \text{in } \R^n \setminus \overline\Omega  \, ,
\\
 f \,&\, = \,&\, 0 & \text{on } \de \Omega \, ,
\\
f(x) \,&\,\to\,&\, 1 & \text{as } \abs{x} \to +\infty \, .
\end{array} \right.
\end{equation}
Before proceeding it is important to notice that the second equation in the above system, corresponds to the identity $\Ric_{g_{\R^n}}=0$, that was hidden in the original formulation of problem~\eqref{eq:system}. In order to simplify the references for the forthcoming analysis, we observe that in the new conformal framework the counterpart of Definition~\ref{ass:main} is given by the following
\begin{defi}
  \label{ass:fg}
  A triple $(\Omega,f,g)$ is said to be a {\em regular solution} to system~\eqref{eq:geom-system} if
  \begin{itemize}
  \item $\Omega  \subset \R^n$, $n \geq 3$, is a bounded open domain. Without loss of generality, we also assume that $\Omega$ contains the origin and $\R^n\setminus\Omega$ is connected.
  \item The function
  $f \colon \R^n \setminus \Omega \to \coint{0}{1}$ and the Riemannian
  metric $g$ defined on $\R^n \setminus \Omega$ satisfy system~\eqref{eq:geom-system}. The function $f$ and the coefficients of the metric $g$ --whenever the latter are expressed with respect to standard Euclidean coordinates-- are $\C^{\infty}$ in $\R^n\setminus\Omega$, and analytic in $\R^n\setminus\overline{\Omega}$.
\end{itemize}
\end{defi}
\noindent In order to make the new formulation more familiar, we observe that taking the trace of the second equation in~\eqref{eq:geom-system} one finds that the scalar curvature ${\rm R}$ of $g$ verifies
\begin{equation}
\label{eq:R}
{\rm R}\,=\,n(n-1)\,\frac{\abs{\nabla f}_g^2}{1-f^2}\,.
\end{equation}
As already noticed, when $\Omega$ is a ball, we have that $g$ is the spherical metric, hence ${\rm R}$ is constant on the whole $\R^n\setminus\Omega$. This suggest that, when looking for rigidity results, it is natural to investigate under which conditions one can conclude that the quantity 
$$
\frac{\abs{\nabla f}_g^2}{1-f^2} \,\, = \,\, u^{-2\frac{n-1}{n-2}}\left| \frac{Du}{n-2} \right|^{2} 
$$ 
is constant. In fact, this quotient will play the role of what is called a $P$-function in the literature about overdetermined boundary value problems.


\subsection{Outline of the proof}

\label{sub:outline}

We are now in the position to outline the analysis of the  regular solutions of system~\eqref{eq:geom-system}, focussing on the key points. This gives us the opportunity to present some intermediate results that we believe of some independent interest. It is important to keep in mind that these results are based on Definition~\ref{ass:fg} and have no ``knowledge'' of the original function $u$.

Combining the Bochner formula with the equations in~\eqref{eq:geom-system}, we prove  in Section~\ref{sec:estimates}
that the vector field
\begin{equation}
  \label{eq:def-x}
  X = \frac{1}{(1-f){(1+f)}^{n-1}} \left( \nabla \abs{\nabla f}_g^2 - \frac{2}{n} \Delta_g f \nabla f \right) = \frac{1}{{(1+f)}^{n-2}}  \nabla\left( \frac{\abs{\nabla f}_g^2}{1-f^2}\right)
\end{equation}
has nonnegative divergence. This fact admits two different
interpretations, which we leverage to obtain two different classes of
results, one based on the divergence theorem and the other based on the strong maximum principle.

\medskip

\noindent {\bf Monotonicity via divergence theorem.}
On one hand, one can integrate $\Div_g X$ between two different level sets
of $f$ and obtain by Stokes' theorem that the quantity
$\int_{\{f=s\}} \scal{X}{\nu_g}_g \d \sigma_g$ is monotonic. In turns, this fact leads to the monotonicity of the function 
\begin{equation}
  \label{eq:phi}
 [0,1) \, \ni \,  s \,\, \longmapsto \,\, \Phi(s) \, = \,  \frac{1}{{(1-s^2)}^\frac{n+2}{2}} \int_{\{f=s\}} \abs{\nabla f}_g^3 \d \sigma_g \, ,
\end{equation}
which is the content of the following theorem
(notice that the fact that $\Phi$ is well defined is an easy consequence of the analyticity of $f$ and $g$, as we will prove in Section~\ref{sec:integral_id}).
%
%
%
%
%
%
%
%
%
%
\begin{theo}
	\label{theo:main-thm-f}
  Let $(\Omega,f,g)$ be a regular solution to problem~\eqref{eq:geom-system} in the sense of Definition~\ref{ass:fg}. Then $\Phi$ is
  differentiable and, for any $s\in\coint{0}{1}$, it holds
	\begin{equation}
	\label{eq:main-thm-f}
	{(1-s^2)}^\frac{n+2}{2} \Phi'(s)= -2 \left[ \int_{\{f=s\}} \abs{\nabla f}_g^2 H_g \d \sigma_g +(n-1)\frac{s}{1-s^2} \int_{\{f=s\}} \abs{\nabla f}_g^3 \d \sigma_g \right] \leq 0 \, ,
	\end{equation}
where $H_g$ is the mean curvature of the hypersurface $\{f=s\}$ computed with respect
to the normal pointing towards $\{f \geq s\}$ and the
metric $g$.	In fact, the left hand side is increasing in $s$ and it tends to zero as $s \to 1$. 
Moreover, $\Phi'(s)=0$ for some $s\in\coint{0}{1}$ if and only if $\Omega$ is a ball
and $f$ is rotationally symmetric.
\end{theo}

\noindent Translating Theorem~\ref{theo:main-thm-f} back in terms of $u$ and $g_{\R^n}$, one obtains the following result.

\begin{theo}
	\label{theo:main-thm}
	Let $(\Omega,u)$ be a regular solution to problem~\eqref{eq:system} in the sense of Definition~\ref{ass:main}. Then, for all $t\in\ocint{0}{1}$, it holds
	\begin{equation}
	\label{eq:main-thm}
	\int_{\{u=t\}} \abs{D u}^2 \left[ \frac{H}{n-1} - \frac{\abs{D u}}{(n-2)u} \right] \d \sigma \geq 0 .
	\end{equation}
Moreover, equality is fulfilled for some $t\in\ocint{0}{1}$ if and only if $\Omega$ is a ball
	and $u$ is rotationally symmetric.
\end{theo}

\noindent The proof of formula~\eqref{eq:main-thm-f} and its equivalence with~\eqref{eq:main-thm} is treated in
Section~\ref{sec:integral_id}. The equality case is shown to imply that $\Div_g X = 0$,
which in turn implies that $\nabla^2 f = \lambda g$, for a certain
function $\lambda \in \C^\infty(M)$. This can be used as
the start to trigger a rigidity argument, which is fully discussed in
Section~\ref{sec:equality}. It should be noticed that Theorem~\ref{theo:main-thm} is not a new result. The original proof appeared in~\cite{agostiniani-mazzieri-14} as an application of the {\em cylindrical ansatz}, and a second proof of the same result, based on some basic inequalities for symmetric functions, is shown in~\cite{bianchini-ciraolo}. 
We also remark that Theorem~\ref{theo:main-thm} has been vastly extended in~\cite{agostiniani-mazzieri}. In fact, our function $\Phi(s)$, rewritten in terms of $u$ and $g_{\R^n}$, coincides, up to a constant, with the function $U_3(t)$ defined as in~\eqref{eq:U_p}. As stated in the introduction, in~\cite{agostiniani-mazzieri} the {\em cylindrical ansatz} is used to prove the monotonicity of an entire family of functions $U_p(t)$ indexed on $p\geq 2-1/(n-1)$, while it seems that the {\em spherical ansatz} helps to recover the monotonicity only for $p=3$.

\medskip

\noindent {\bf Rigidity via strong maximum principle.} On the other hand, the divergence of $X$ can be seen as a positive definite elliptic differential operator of the second order in
divergence form, applied to the function
\begin{equation*}
\frac{\abs{\nabla f}_g^2}{1-f^2} \,\,=\,\,\frac{\rm R}{n(n-1)}\, .
\end{equation*}
More explicitly, recalling the definition of $X$ given in~\eqref{eq:def-x}, the inequality $\Div_g X\geq 0$ translates in the following elliptic inequality for the scalar curvature ${\rm R}$ of $g$
\begin{equation}
\label{eq:ellin-R}
\Delta_g {\rm R} \,-\,(n-2)\Scal{\frac{\nabla f}{1+f}}{\nabla {\rm R}}_{\!g}\,\,\geq\,\, 0\,.
\end{equation}
This is the starting point of our second approach, that exploits a peculiar feature of the {\em spherical ansatz}, namely the one point compactification.
In fact, the present geometric construction allows one to compactify the space $\R^n\setminus\Omega$ 
adding the north pole $N$ to its preimage via the stereographic projection. Before setting up some notations, we recall from Remark~\ref{rem:different-radiai} that, in the rigidity statements, we ultimately expect $\R^n\setminus\overline{\Omega}$ to be an hemisphere of radius $\rho={[\Ca(\Omega)]}^{{1}/({n-2})}$, so that it is natural to consider the stereographic projection
\begin{equation*}
  \pi_\rho \colon S_\rho^n \setminus \{N\} \,, \longrightarrow \R^n \qquad
  (z', z^{n+1}) \longmapsto \frac{\rho \,z'}{\rho-z^{n+1}} \,  .
\end{equation*}
where $S_\rho^n$ is the standard sphere of radius $\rho={[\Ca(\Omega)]}^{{1}/({n-2})}$ and $N =  (0,\ldots , 0, \rho) \in \R^{n+1}$ is its north pole.
The space
\begin{equation}
\label{eq:M}
M \,\, = \,\, \pi_{\rho}^{-1}(\R^n \setminus \Omega)\cup\{N\} \, \subset \, S_\rho^{n} \, 
\end{equation}
is now compact. 
On $M$ we define the function $\phi:M\to\R$ and the metric $\gamma$ respectively as the pull-back of $f$ and $g$, suitably extended through the north pole. More precisely, we set
\begin{equation}
\label{eq:phi-gamma}
\def\arraystretch{1.5}
\phi(y)\,=\,
\left\{ \begin{array}{ll}
f\circ\pi_\rho(y) & \mbox{if }y\in\pi_\rho^{-1}(\R^n \setminus\Omega)\,,
\\
1 & \mbox{if }y=N\,,
\end{array}\right.
\qquad\quad
\gamma_{|_{y}}\,=\,
\left\{ \begin{array}{ll}
{(\pi_\rho^* g)}_{|_{y}} & \mbox{if }y\in\pi_\rho^{-1}(\R^n \setminus\Omega)\,,
\\
{(g_{S_\rho^n})}_{|_N} & \mbox{if }y=N\,.
\end{array}\right.
\end{equation}
Using the asymptotic estimates on the function $f$ and its derivative, it is easy to deduce (see Section~\ref{sec:proof-part2}) that $\phi \in \C^2(M)$, whereas the metric $\gamma$ has $\C^2$ regularity in a neighborhood of the north pole. In particular, its scalar curvature ${\rm R}_\gamma$ is well-defined and continuous on the whole $M$. It is then possible to extend the validity of~\eqref{eq:ellin-R} at the north pole. In fact, we will show that ${\rm R}_\gamma$ satisfies the elliptic inequality
\begin{equation}
\label{eq:ellin-Rgamma}
\Delta_\gamma {\rm R}_\gamma \,-\,(n-2)\Scal{\frac{\nabla \phi}{1+\phi}}{\nabla {\rm R}_\gamma}_{\!\gamma}\,\,\geq\,\, 0
\end{equation}
in the $\mathcal{W}^{1,2}(M)$--sense.
This allows to invoke a strong maximum principle, which leads to the following result.	
	
\begin{theo}
  \label{theo:main-thm2-f}
Let $(\Omega,f,g)$ be a regular solution to problem~\eqref{eq:geom-system} in the sense of Definition~\ref{ass:fg}, and let $(M,\gamma)$ be defined by~\eqref{eq:M} and~\eqref{eq:phi-gamma}, respectively. Then, for every
  $y \in M$, it holds
  \begin{equation} 
  \label{eq:main-thm2-f}
  {\rm R}_\gamma(y) \leq \max_{\de M} {\rm R}_\gamma \,. 
  \end{equation}
Moreover, if the equality is fulfilled for some $y \in M\setminus\de M=\pi^{-1}(\R^n \setminus \overline{\Omega})\cup\{N\}$, then $(M,\gamma)$ is isometric to an hemisphere of radius ${[\Ca(\Omega)]}^{1/(n-2)}$.
\end{theo}

\noindent Translating Theorem~\ref{theo:main-thm2-f} in terms of $u$ and $g_{\R^n}$, we obtain the following extended version of Theorem~\ref{theo:main-thm2} in the introduction.

\begin{theo}
\label{theo:main-thm-ext2}
Let $(\Omega,u)$ be a regular solution to problem~\eqref{eq:system} in the sense of Definition~\ref{ass:main}. Then, the following two inequalities hold
\begin{align}
\label{eq:theo-ext-1}
	\frac{1}{\Ca(\Omega)} \,\leq\,\left({\max_{\de\Omega}\left|\frac{Du}{n-2}\right|}\right)^{n-2} \quad \hbox{and} \qquad
{\frac{|D u|}{ u^{\frac{n-1}{n-2}}}}(x)\,\leq\,{\max_{\de\Omega}\left|{Du}\right|} \, ,
\end{align}
for every $x \in \R^n \setminus \overline{\Omega}$. Moreover, if the equality is fulfilled in either the first or the second inequality, for some $x \in \R^n \setminus \overline{\Omega}$, then $\Omega$ is a ball
	and $u$ is rotationally symmetric.
\end{theo}

\noindent The detailed proof of inequality~\eqref{eq:main-thm2-f} and its equivalence with formul\ae~\eqref{eq:theo-ext-1} in Theorem~\ref{theo:main-thm-ext2} can be found in
Section~\ref{sec:proof-part2}. 
Again, the equality case is shown to imply that $\nabla^2 f = \lambda g$ for a certain
function $\lambda$, and this fact is used in
Section~\ref{sec:equality} to prove the rigidity statements.

\section{The conformally equivalent formulation}
\label{sec:conformal-form}

We will shortly detail the study of the conformally equivalent
formulation of the problem that we proposed above. Before that, let us
recall a number of notational conventions.

Given $\Omega \subset \R^n$ bounded open set with smooth boundary and
containing the origin, consider the unique solution of system~\eqref{eq:system}, then build the metric $g$ and the function
$f$ as in~\eqref{eq:def-g} and~\eqref{eq:def-f}. In the following we
will use the same conventions as in~\cite{agostiniani-mazzieri-14},
using the symbols $\scal{\cdot}{\cdot}$, $\abs{\cdot}$, $D$, $D^2$ and
$\Delta$ respectively as the scalar product, norm, covariant
derivative, Hessian and Laplacian with respect to the metric
$g_{\R^n}$ and $\scal{\cdot}{\cdot}_g$, $\abs{\cdot}_g$, $\nabla$,
$\nabla^2$ and $\Delta_g$ as the same objects with respect to the
metric $g$. For the metric $g$ we also define $\Gamma_{ij}^k$ and
$\Ric_{ij}$ to be the Christoffel symbols and the Ricci tensor, whose
sign is chosen so that the standard sphere has positive scalar
curvature. The notations $Du$ and $\nabla f$ will also denote the
gradient vector of the corresponding functions, i.e. the composition
of the covariant derivative with the inverse of the metric.

\smallskip

Once the framework is set up, all we have to do is to compute the
relationships between the geometric objects built on $g_{\R^n}$ and
$u$ and those built on $g$ and $f$, which we shall readily do. The
reader wishing to read interesting mathematical content is suggested
to skip to the next section and come back here only to verify formul\ae\ 
as they desire so.

\medskip

\noindent{\bf Conformal change and derivation of system~\eqref{eq:geom-system}.}
From~\eqref{eq:def-f} we have that
\[ u^{-\frac{2}{n-2}} = \frac{1+f}{1-f} . \]
Rewriting $g$ in terms of $f$ therefore gives
\begin{equation}
  \label{eq:g-from-f}
  g = \left( \frac{2}{1+\frac{1+f}{1-f}} \right)^2 g_{\R^n} = {(1-f)}^2 g_{\R^n} ,
\end{equation}
from which it is easy to compute the Christoffel symbols:
\[ \Gamma_{ij}^k = -\frac{1}{1-f} (D_i f \delta_j^k + D_j f \delta_i^k - D_h f g_{\R^n}^{kh} g^{\R^n}_{ij} ) . \]
Formul\ae\ for Hessian and Laplacian of a smooth function $w \colon \R^n
\setminus \overline{\Omega} \to \R$ follow at once:
\begin{align*}
  \nabla^2_{ij} w & = D^2_{ij} w - \Gamma_{ij}^k D_k w \\
  & = D^2_{ij} w + \frac{1}{1-f} \left( D_i f D_j w + D_j f D_i w - \scal{Df}{Dw} g^{\R^n}_{ij}\right) , \\
  \Delta_g w & = g^{ij} \nabla^2_{ij} w \\
  & = \frac{1}{{(1-f)}^2} \Delta w - \frac{n-2}{{(1-f)}^3} \scal{Df}{Dw} .
\end{align*}
Computing the covariant derivatives of $f$ (and recalling that $\Delta
u = 0$) we obtain
\begin{align*}
  Df & = -\frac{4}{n-2} \cdot \frac{u^{-\frac{n}{n-2}}}{{(u^{-\frac{2}{n-2}}+1)}^2} Du \\
  & = -\frac{1}{n-2} {(1+f)}^\frac{n}{2} {(1-f)}^{-\frac{n-4}{2}} Du \\
  & = -\frac{1}{n-2} \left( \frac{1+f}{1-f} \right)^{\frac{n-2}{2}} (1-f^2) Du , \\
  D^2 f & = \left( \frac{n}{2} \cdot \frac{1}{1+f} + \frac{n-4}{2} \cdot \frac{1}{1-f} \right) Df \otimes Df - \frac{1}{n-2} \left( \frac{1+f}{1-f} \right)^{\frac{n-2}{2}} (1-f^2) D^2u , \\
  \Delta f & = \left( \frac{n}{2} \cdot \frac{1}{1+f} + \frac{n-4}{2} \cdot \frac{1}{1-f} \right) \abs{Df}^2 \\
  & = \frac{1}{1-f^2} \cdot (n-2-2f) \abs{Df}^2,
\end{align*}
for the Euclidean case and
\begin{align*}
  \nabla^2_{ij} f & = D^2_{ij} f - \Gamma_{ij}^k D_k f \\
  & = D^2_{ij} f + \frac{1}{1-f} (2 D_i f D_j f - \abs{Df}^2 g_{ij}^{\R^n}) \\
  & = \left( \frac{n}{2} \cdot \frac{1}{1+f} + \frac{n-\cancel{4}}{2} \cdot \frac{1}{1-f} \right) D_i f D_j f - \frac{1}{n-2} \left( \frac{1+f}{1-f} \right)^{\frac{n-2}{2}} (1-f^2) D^2_{ij} u \\
  & \qquad {} + \frac{1}{1-f} (\cancel{2 D_i f D_j f} - \abs{Df}^2 g^{\R^n}_{ij} ) \\
  & = \frac{n}{1-f^2} D_i f D_j f - \frac{1}{1-f} \abs{Df}^2 g^{\R^n}_{ij} - \frac{1}{n-2} \left( \frac{1+f}{1-f} \right)^{\frac{n-2}{2}} (1-f^2) D^2_{ij} u  \\
  & = \frac{4}{(n-2)^2} \, \frac{u^{-2\frac{n-1}{n-2}}}{(u^{-\frac{2}{n-2}}+1)^2}\left[n D_i u D_j u - 2 \frac{u^{-\frac{2}{n-2}}}{u^{-\frac{2}{n-2}}+1}\abs{D u}^2 g_{ij}^{\R^n}- (n-2) u D^2_{ij} u \right],
  \\
  \Delta_g f & = \frac{1}{{(1-f)}^2} \Delta f + \frac{1}{{(1-f)}^3} (2-n) \abs{Df}^2 \\
  & = (-nf) \frac{\abs{Df}^2}{{(1-f)}^2(1-f^2)} \\
  & = (-nf) \frac{\abs{\nabla f}_g^2}{1-f^2}\,,
\end{align*}
for the geometrical case (of course $\nabla f = Df$, as covariant
derivation of functions does not actually depend on the metric; however
$\abs{\nabla f}^2_g = {(1-f)}^{-2} \abs{Df}^2$ and, in particular,
$\abs{\nabla f}_g^2\,g = \abs{Df}^2\, g_{\R^n}$). In Corollary~\ref{cor:est-f} the above relations are used to derive
asymptotic estimates for $f$.  For future convenience, we also write
down the relation between the norm of the gradient of $f$ and $u$
\begin{align}
\frac{\abs{\nabla f}_g^2}{1-f^2}
&=\frac{\abs{D f}^2}{(1-f)^2(1-f^2)} \nonumber
\\
&=\frac{1}{(n-2)^2}\left(\frac{1+f}{1-f}\right)^{n-1}\abs{D u}^2 \nonumber
\\
\label{eq:|naf|}
&=\,u^{-2\left(\frac{n-1}{n-2}\right)}\,\abs*{\frac{D u}{n-2}}^2 ,
\end{align}
from which the following estimate is obtained for $\abs{x}\to\infty$:
\begin{equation}
\label{eq:est_f2}
\frac{\abs{\nabla f}_g^2}{1-f^2}={[\Ca(\Omega)]}^{-\frac{2}{n-2}}+o(1) \,.
\end{equation}
Another equation that we need is the one that relates the Ricci tensor of the metric
$g$ with the hessian of the function $f$. From~\cite[Theorem 1.159]{besse} (who uses the opposite
convention as us for the sign of the Laplacian) we have that the
metric $g = e^{2\phi} g_{\R^n}$ has Ricci tensor equal to
\[ -(n-2) \left(D^2 \phi - D \phi \otimes D \phi\right) + \left[-\Delta \phi - (n-2) \abs{D\phi}^2\right] g_{\R^n} . \]
In our case $\phi = \log(1-f)$, so
\begin{align*}
  \Ric & = \frac{n-2}{1-f} \cdot D^2 f + \frac{2(n-2)}{{(1-f)}^2} \cdot Df \otimes Df + \frac{1}{1-f} \cdot \Delta f g_{\R^n} - \frac{n-3}{{(1-f)}^2} \cdot \abs{Df}^2 g_{\R^n} \\
  & = \frac{n-2}{1-f} \nabla^2 f + \frac{1}{1-f} \Delta f\, g_{\R^n} - \frac{1}{{(1-f)}^2} \abs{Df}^2 g_{\R^n} \\
  & = \frac{n-2}{1-f} \nabla^2 f + \frac{\abs{Df}^2}{{(1-f)}^2} \left( \frac{n-2-2f}{1+f} - 1 \right) g_{\R^n} \\
  & = \frac{n-2}{1-f} \nabla^2 f + \frac{n-1-f}{1-f} \cdot \frac{\abs{\nabla f}_g^2 }{1-f^2}\, g \,.
\end{align*}
In particular, we have recovered system~\eqref{eq:geom-system} stated
in the introduction.

\medskip

\noindent{\bf Relevant geometric quantities on the level sets. }We pass now to analyze the geometry of the level sets of $u$ and $f$. Let us start from the case in which our level set is regular, for either $f$ or $u$ (we have already noticed that the
two notions are equivalent). We define the normal vector fields
$\nu$ and $\nu_g$, second fundamental forms $h$ and $h_g$ and the mean
curvatures $H$ and $H_g$ (already mentioned in the introduction) with
respect to the two metrics. The normal vectors are chosen to be directed outwards, namely, on a level set $\{u=t\}=\{f=s\}$, both $\nu$ and $\nu_g$ point towards the exterior set $\{u \leq t\}=\{f \geq s\}$. The second
fundamental forms are chosen with respect to said normal vectors and
stipulating that the standard sphere immersed in the Euclidean space
has positive mean curvature with respect to the normal pointing away
from the origin. The exact formul\ae\ are (in this case $i$ and $j$ only
run along the tangential coordinates to the level set):
\begin{align*}
\nu & = -\frac{Du}{\abs{Du}} \,, & h_{ij} & = \scal{D_i \nu}{\de_j}\,,  & H & = g_{\R^n}^{ij} h_{ij} \, , \\
\nu_g & = \frac{\nabla f}{\abs{\nabla f}_g}\, , & {(h_g)}_{ij} & = \scal{\nabla_i \nu_g}{\de_j}_g , & H_g & = g^{ij} {(h_g)}_{ij} \, .
\end{align*}
If $\{f = s\}$ is a regular level set and $p$ is one of its points,
then the function $f$ is a local coordinate, which can be completed to
a full set of local coordinates $(f, x^2, \dots, x^n)$ around $p$ such
that at any point of the level set the vector fields $\de_{x^k}$ are
tangent to the level set for $k = 2, \dots, n$ and $\de_f$ is orthogonal to it;
the coordinates can also be taken to be all orthogonal to each other at the single point
$p$. The coordinates $x^2$, \dots, $x^n$ are first chosen as the usual normal coordinates on $\{f=s\}$
and then extended to a small open set around $p$ through the flow of
the vector field $\nabla f$. The following relationship is then valid
at any point of $\{f=s\}$:
\begin{equation}
\label{eq:rel-normal}
\nabla f = \abs{\nabla f}_g \cdot \nu_g = \abs{\nabla f}_g^2 \cdot \de_f ,
\end{equation}
since $\de_f$ and $\nabla f$ are by construction parallel and

\begin{align*}
df(\de_f) & = \de_f(f) = 1 \\
df(\nabla f) & = \scal{\nabla f}{\nabla f}_g = \abs{\nabla f}_g^2 \,.
\end{align*}
With this choice of coordinates, the metric can be locally written as
\[ \abs{\nabla f}_g^{-2} df \otimes df + g_{\alpha\beta}(f,x)
dx^\alpha \otimes dx^\beta , \]
where Greek indices range over tangential coordinates. In particular
\begin{equation}
\label{eq:coarea}
\abs{\nabla f}_g \cdot d v_g = d \sigma_g \cdot df\, .
\end{equation}
where $v_g$ is the measure induced by $g$ on $\R^n\setminus\Omega$ and $\sigma_g$ is the measure induced by $g$ on the level set $\{f=s\}$. For a regular level set $\{f=s\}$ we also have
\begin{align}
\!\!\!{(h_g)}_{ij} & = \scal{\nabla_i \nu_g}{\de_j}_g = \Scal{\nabla_i \frac{\nabla f}{\abs{\nabla f}_g}}{\de_j}_{\!\!g} =  \frac{1}{\abs{\nabla f}_g} \scal{\nabla_i \nabla f}{\de_j}_g + \de_i \frac{1}{\abs{\nabla f}_g} \cancel{\scal{\nabla f}{\de_j}}_g = \frac{\nabla^2_{ij} f}{\abs{\nabla f}_g} , \nonumber \\
H_g & = g^{ij} {(h_g)}_{ij} = \frac{\Delta_g f - g^{ff} \nabla^2_{ff} f}{\abs{\nabla f}_g}  = \frac{\Delta_g f - \nabla^2_{\nu\nu} f}{\abs{\nabla f}_g} = \frac{\Delta_g f}{\abs{\nabla f}_g} - \frac{\nabla^2 f(\nabla f, \nabla f)}{\abs{\nabla f}_g^3} \,, \label{eq:mean-curv}
\end{align}
so
\begin{equation}
\label{eq:nabla2-h}
\nabla^2 f(\nabla f, \nabla f) = \abs{\nabla f}_g^2 \Delta_g f - \abs{\nabla f}_g^3 H_g = -\frac{nf}{1-f^2} \abs{\nabla f}_g^4 - \abs{\nabla f}_g^3 H_g \,.
\end{equation}
The corresponding formula for $g_{\R^n}$, thanks to the harmonicity of
$u$, simplifies to
\[ D^2u(Du,Du) = \abs{Du}^3 H \,  . \]
Let us now pass to the analysis of the nonregular level sets of $u$ and $f$. First of all, since $u$ is harmonic, it follows  from~\cite[Theorem~1.7]{Har_Sim} (see also~\cite{Lin}) that the $(n-1)$-dimensional Hausdorff measure of the level sets of $u$ (thus also of $f$) is locally finite. The properness of $u$ and $f$ then implies that all the level sets of $u,f$ have finite $\mathcal{H}^{n-1}$-measure.
Moreover, from~\cite[Theorem~1.17]{Che_Nab_Val} we know that the Minkowski dimension of the set of the critical points of $u$ is bounded above by $(n-2)$. In particular, also the Hausdorff dimension is bounded by $(n-2)$ (the bound on the Hausdorff dimension was proved before, see for instance~\cite{Har_Hof_Hof_Nad}). This implies that, for all $t\in\ocint{0}{1}$, the normal $\nu$ to the level set $\{u=t\}$ is defined $\mathcal{H}^{n-1}$-almost everywhere on $\{u=t\}$, and, consequently, so are the second fundamental form $h$ and mean curvature $H$. Since the critical points of $f$ are the same as the ones of $u$ (as one deduces immediately from~\eqref{eq:|naf|}), the same is true for the normal $\nu_g$, the second fundamental form $h_g$ and the mean curvature $H_g$ with respect to $g$. In particular, formul\ae~\eqref{eq:coarea},~\eqref{eq:mean-curv} and~\eqref{eq:nabla2-h} hold $\mathcal{H}^{n-1}$-almost everywhere on any level set of $f$.

\section{Pointwise identities via Bochner formula}
\label{sec:estimates}

In this section we will abstract completely from the original
formulation of the problem (which is system~\eqref{eq:system}) and
work only with triples $(\Omega,f,g)$ satisfying system~\eqref{eq:geom-system}, that we recall here:
\begin{equation}
\label{eq:geom-system2}
\def\arraystretch{1.8}
\left\{ \begin{array}{r@{}c@{}ll}
\Delta_g f \,&\,=\, &\,-n \left(\frac{\abs{\nabla f}_g^2}{1-f^2} \right) \, f & \text{in } \R^n \setminus \overline\Omega \, ,
\\
\Ric \,&\, = \,&\, \frac{n-2}{1-f} \,\, \nabla^2 f \,\, + \,\, \frac{n-1-f}{1-f} \, \left(\frac{\abs{\nabla f}_g^2 }{1-f^2}\right) \,g & \text{in } \R^n \setminus \overline\Omega  \, ,
\\
f \,&\, = \,&\, 0 & \text{on } \de \Omega \, ,
\\
f(x) \,&\,\to\,&\, 1 & \text{as } \abs{x} \to +\infty \, .
\end{array} \right.
\end{equation}
We also recall that, in the spherical setting,
we ultimately expect that the function $$
x\,\,\longmapsto\,\,\frac{\abs{\nabla f}_g^2}{1-f^2}(x)
$$
will be relevant in the rigidity statements, as we need to prove it is
constant. Also, we expect the metric $g$ to be
isometric to the north round half-sphere deprived of the north pole, where the
function $f$ plays the role of the $z^{n+1}$ coordinate (consistently,
it has value zero on $\de\Omega$, which will be the equator, and goes
to $1$ at infinity, which will be the north pole).

We will make use of the well-known Bochner formula, that we repeat
here for the convenience of the reader: for any smooth function $f$ on
a Riemannian manifold there holds
\[ \frac{1}{2} \Delta_g \abs{\nabla f}_g^2 = \abs{\nabla^2 f}_g^2 + \Ric(\nabla f, \nabla f) + \scal{\nabla \Delta_g f}{\nabla f}_g . \]
The formula is easy to prove rearranging and commuting derivatives,
but see also~\cite[Proposition~4.15]{gallot}. In our case, using~\eqref{eq:geom-system2} the Bochner formula rewrites as
\[ \frac{1}{2} \Delta_g \abs{\nabla f}_g^2 = \abs{\nabla^2 f}_g^2 + \frac{n-2}{1-f} \nabla^2 f(\nabla f, \nabla f) + \frac{n-1-f}{1-f} \cdot \frac{\abs{\nabla f}_g^4}{1-f^2} + \scal{\nabla \Delta_g f}{\nabla f} . \]
We can exploit this formula to find a monotonic quantity for the
function $f$. Let us consider the vector field $X$ which was defined
in~\eqref{eq:def-x} and which we report here again for the reader's
convenience:
\begin{equation}
  \label{eq:def-x2}
  X = \frac{1}{(1-f){(1+f)}^{n-1}} \left(\nabla \abs{\nabla f}_g^2 - \frac{2}{n} \Delta_g f \nabla f \right) = \frac{1}{{(1+f)}^{n-2}}  \nabla\left( \frac{\abs{\nabla f}_g^2}{1-f^2}\right) .
\end{equation}

\begin{lem}
	\label{le:div_X}
  Let $(\Omega,f,g)$ be a regular solution to problem~\eqref{eq:geom-system2} in the sense of Definition~\ref{ass:fg}, and let $X$ be the vector field defined by~\eqref{eq:def-x2}.  Then $\Div_g X \geq 0$, where
  $\Div_g$ denotes the divergence with respect to the metric $g$ (in
  coordinates: $\Div_g X = \nabla_i X^i$). Moreover, at any point,
  $\Div_g X = 0$ if and only if $\nabla^2 f = \lambda g$ for some real
  number $\lambda$.
\end{lem}

\begin{proof}
  As before, this is just a computation:
\begin{align*}
  \Div_g X & = \frac{1}{(1-f){(1+f)}^{n-1}} \left[ \Delta_g \abs{\nabla f}_g^2 - \frac{2}{n} {(\Delta_g f)}^2 - \frac{2}{n} \scal{\nabla \Delta_g f}{\nabla f}_g \right. \\
    & \qquad \left. - \frac{-nf+n-2}{1-f^2} \Scal{\nabla \abs{\nabla f}_g^2 - \frac{2}{n} \Delta_g f \nabla f}{\nabla f}_g \right] \\
  & = \frac{2}{(1-f){(1+f)}^{n-1}} \left[ \abs{\nabla^2 f}_g^2 - \frac{1}{n} {(\Delta_g f)}^2 - \frac{1}{n} \scal{\nabla \Delta_g f}{\nabla f}_g \right. \\
    & \qquad {}+ \frac{n-2}{1-f} \nabla^2 f(\nabla f, \nabla f) + \frac{n-1-f}{1-f} \cdot \frac{\abs{\nabla f}_g^4}{1-f^2} + \scal{\nabla \Delta_g f}{\nabla f}_g \\
    & \qquad \left. {}- \frac{1}{2} \cdot \frac{-nf+n-2}{1-f^2} \scal{\nabla \abs{\nabla f}_g^2}{\nabla f}_g - \frac{1}{n} \cdot \frac{-nf+n-2}{1-f^2} \Delta_g f \abs{\nabla f}_g^2 \right] .
\end{align*}
To make computations easier to follow, we single out:
\[ \scal{\nabla \Delta_g f}{\nabla f}_g = -n \Scal{\nabla \left( f \frac{\abs{\nabla f}^2_g}{1-f^2} \right)}{\nabla f}_{\!\!g} = -n \left[ \frac{1+f^2}{{(1-f^2)}^2} \abs{\nabla f}^4_g + \frac{f}{1-f^2} \scal{\nabla \abs{\nabla f}_g^2}{\nabla f}_g \right] \]
(using again~\eqref{eq:geom-system2}) and:
\begin{equation}
	\label{eq:distr-nabla2}
	\scal{\nabla\abs{\nabla f}_g^2}{\nabla f}_g = g_{ij} \cdot \nabla_i (g_{k\ell} \nabla_k f \nabla_\ell f) \cdot \nabla_j f = 2 g_{ij} g_{k\ell} \nabla^2_{ik} f \cdot \nabla_\ell f \cdot \nabla_j f = 2 \nabla^2 f(\nabla f, \nabla f) .
\end{equation}

So we can conclude:
\begin{align}
  \Div_g X & = \frac{2}{(1-f){(1+f)}^{n-1}} \left[ \abs{\nabla^2 f}_g^2 - \frac{1}{n} {(\Delta_g f)}^2 \right. \nonumber \\
    & \qquad {}- (n-1) \left( \frac{1+f^2}{{(1-f^2)}^2} \abs{\nabla f}^4_g + \frac{2f}{1-f^2} \nabla^2 f(\nabla f, \nabla f) \right) \nonumber \\
    & \qquad {}+ \frac{(n-2)f+n-2}{1-f^2} \nabla^2 f(\nabla f, \nabla f) + \frac{-(2n+1)f^2+(n-2)f-1}{{(1-f^2)}^2} \abs{\nabla f}_g^4 \nonumber \\
    & \qquad \left. - \frac{-nf+n-2}{1-f^2} \nabla^2 f(\nabla f, \nabla f) - \frac{-nf+n-2}{1-f^2} \cdot \frac{f}{1-f^2} \abs{\nabla f}_g^4 \right] \nonumber \\
  & = \frac{2}{(1-f){(1+f)}^{n-1}} \left[ \abs{\nabla^2 f}_g^2 - \frac{1}{n} {(\Delta_g f)}^2 \right] \geq 0 . \label{eq:div-x}
\end{align}
The last inequality follows from the arithmetic-quadratic mean inequality,
since if $\lambda_1$, \dots $\lambda_n$ are the eigenvalues of
$\nabla^2 f$ (i.e., the principal curvatures), then
$\abs{\nabla^2 f}_g^2 = \sum_i \lambda_i^2$ and
${(\Delta_g f)}^2 = \left( \sum_i \lambda_i \right)^2$. In particular,
when equality is fulfilled all $\lambda_i$ are identical and
$\nabla^2 f = \lambda g$.
\end{proof}

\noindent We conclude this section showing that the inequality $\Div_g X$ can be rewritten as a second order elliptic inequality for the scalar curvature ${\rm R}$ of $g$. To this end, recalling the definition of $X$ given in~\eqref{eq:def-x2} we compute
\begin{align*}
\Div_g X\,&=\,\Div_g\left[\frac{1}{{(1+f)}^{n-2}}  \nabla\left( \frac{\abs{\nabla f}_g^2}{1-f^2}\right)\right]
\\
&=\,\frac{1}{{(1+f)}^{n-2}} \left[ \Delta_g\left( \frac{\abs{\nabla f}_g^2}{1-f^2}\right)-(n-2)\Scal{\frac{\nabla f}{1+f}}{\nabla\left( \frac{\abs{\nabla f}_g^2}{1-f^2}\right)}_{\!\!g}\right]
\end{align*}
Since we have already observed in~\eqref{eq:R} that the scalar curvature ${\rm R}$ is a multiple of $\abs{\nabla f}_g^2/(1-f^2)$, the inequality $\Div_g X$ can be rewritten as
\begin{equation}
\label{eq:ellin-R2}
\Delta_g{\rm R}-(n-2)\Scal{\frac{\nabla f}{1+f}}{\nabla{\rm R}}_{\!g}\,\geq\,0\,.
\end{equation}
Therefore, ${\rm R}$ satisfies a second order elliptic inequality, as anticipated. Inequality~\eqref{eq:ellin-R2} is the starting point of the maximum principle argument that will be discussed in Section~\ref{sec:proof-part2}.

\section{Integral identities}
\label{sec:integral_id}

	In this section we will study the properties of the function $\Phi$, which was defined in Subsection~\ref{sub:outline} and which we recall here:
\begin{equation}
\label{eq:Phi}
\Phi(s) = \frac{1}{{(1-s^2)}^\frac{n+2}{2}}\int_{\{f=s\}} \abs{\nabla f}^3_g \d \sigma_g .
\end{equation}
More precisely, we will show that $\Phi$ is differentiable and monotonic for all $s\in\coint{0}{1}$. 
As a byproduct of this analysis, we will prove inequality~\eqref{eq:main-thm-f} and
its equivalence with~\eqref{eq:main-thm}. 

We point out that the proof of the results presented in the Introduction do not depend on the content of this section. However, it is our opinion that the computations here should help the reader to compare the {\em spherical ansatz} with  the {\em cylindrical ansatz}, shading some lights on the advantages and disadvantages of both methods.

\medskip

\noindent{\bf Well-definedness and continuity.}
First of all, let us notice that $\Phi(s)$ is indeed well-defined for all values $s\in[0,1)$. 

\begin{lem}
	\label{le:Phi-wd}
	Let $(\Omega,f,g)$ be a regular solution to problem~\eqref{eq:geom-system2} in the sense of Definition~\ref{ass:fg}. Then the function $\Phi:[0,1)\to\R$ introduced in~\eqref{eq:Phi} is well-defined and it holds $\limsup_{s\to 1^-}\Phi(s)<+\infty$.
\end{lem}

\begin{proof}
We have already observed in Section~\ref{sec:conformal-form} that, from the analyticity and properness of $f$, it follows that the level sets of $f$ have finite $(n-1)$-dimensional Hausdorff measure. Moreover, from Corollary~\ref{cor:asympt}, it follows that $|\nabla f|_g\to 0$ as $|x|\to\infty$. In particular, since $|\nabla f|_g$ is smooth, it is uniformly bounded on $\R^n\setminus\Omega$. This proves that $\Phi(s)$ is well-defined. 

In order to prove that it is bounded as $s\to 1^-$, we use Corollaries~\ref{cor:est-f} and~\ref{cor:asympt} to write
$$
\Phi(s)\,<\,C_1\int_{\{f=s\}} \!\abs{x}^{n-1}\d\sigma_g\,<\,C_2\int_{\{f=s\}} \!(1-s)^{-\frac{n-1}{2}}\d\sigma_g
$$
for $0<s<1$ sufficiently close to $1$ and for some constants $0<C_1<C_2$. The bound on $\Phi(s)$ as $s\to 1^-$ now follows from estimate~\eqref{eq:est-levels-f}.
\end{proof}

\noindent We pass now to discuss the continuity of $\Phi$.

\begin{lem}
	\label{le:Phi-cont}
Let $(\Omega,f,g)$ be a regular solution to problem~\eqref{eq:geom-system2} in the sense of Definition~\ref{ass:fg}. Then the function $\Phi:[0,1)\to \R$ defined by~\eqref{eq:Phi} is continuous and, for all $s\in\coint{0}{1}$, it holds
\begin{equation}
\label{eq:Phi-id1}
(1-s^2)\Phi(s)\,=\,-\int_{\{f>s\}}\frac{\nabla^2 f(\nabla f,\nabla f)}{2(1-f^2)^{\frac{n}{2}}}\d v_g\,.
\end{equation}
\end{lem}

\begin{proof}
We start with the following computation:
\begin{align}
\notag
\Div_g\left[\frac{|\nabla f|_g^2}{(1-f^2)^{\frac{n}{2}}}\cdot\nabla f\right]\,&=\,
\frac{\Scal{\nabla\abs{\nabla f}_g^2}{\nabla f}_g}{(1-f^2)^{\frac{n}{2}}}
+
\cancel{\frac{|\nabla f|_g^2}{(1-f^2)^{\frac{n}{2}}}\cdot\Delta_g f}
+
\cancel{\frac{nf|\nabla f|_g^4}{(1-f^2)^{\frac{n+2}{2}}}}
\\
\label{eq:div_naf}
&=\,\frac{\nabla^2 f(\nabla f,\nabla f)}{2(1-f^2)^{\frac{n}{2}}}\,,
\end{align}
where the cancellation is a consequence of the second formula in system~\eqref{eq:geom-system2}. 
Now let $0\leq s<1$ be a regular level set of $f$ and let $s<S<1$. It follows from Corollary~\ref{cor:est-f} that if we choose $S$ large enough, then necessarily the level set $\{f=S\}$ is regular.
Integrating formula~\eqref{eq:div_naf} in $\{s<f< S\}$, and applying the divergence theorem, we find
\begin{align}
\notag
\int_{\{s< f<S\}}\frac{\nabla^2 f(\nabla f,\nabla f)}{2(1-f^2)^{\frac{n}{2}}}\d v_g\,&=\,\int_{\{f=S\}}\Scal{\frac{\abs{\nabla f}_g^2}{(1-f^2)^{\frac{n}{2}}}\cdot\nabla f}{\nu_g}_{\!g}\d \sigma_g-\int_{\{f=s\}}\Scal{\frac{\abs{\nabla f}_g^2}{(1-f^2)^{\frac{n}{2}}}\cdot\nabla f}{\nu_g}_{\!g}\d\sigma_g
\\
\notag
&=\,\int_{\{f=S\}}\frac{\abs{\nabla f}_g^3}{(1-f^2)^{\frac{n}{2}}}\d \sigma_g-\int_{\{f=s\}}\frac{\abs{\nabla f}_g^3}{(1-f^2)^{\frac{n}{2}}}\d\sigma_g
\\
\label{eq:int-id1}
&=\,(1-S^2)\Phi(S)-(1-s^2)\Phi(s)\,.
\end{align}
Moreover, from Lemma~\ref{le:Phi-wd}, we know that $\Phi$ is bounded for values sufficiently close to $1$, thus, taking the limit of~\eqref{eq:int-id1} as $S\to 1^-$, we obtain~\eqref{eq:Phi-id1}.

In the case where the level set $\{f=s\}$ is critical, the argument for the proof of~\eqref{eq:Phi-id1} becomes slightly more technical: for every $\epsilon>0$, one lets 
$$
U_\epsilon=B_\epsilon(\text{Crit}(f))=\bigcup_{x\in\text{Crit}(f)}B_\epsilon(x)
$$ 
be the $\epsilon$-neighborhood of the critical points, integrates~\eqref{eq:div_naf} in $\{s<f<S\}\setminus U_\epsilon$, and then applies the divergence theorem. Taking the limit as $\epsilon\to 0$, since the set of the critical points of $f$ has Minkowski dimension bounded above by $(n-2)$ (as observed in Section~\ref{sec:conformal-form}), the set $U_\epsilon$ shrinks fast enough and one recovers~\eqref{eq:int-id1}. For a more careful explanation of the technical details, we address the interested reader to the proof of~\cite[Proposition~4.1]{agostiniani-mazzieri}, which uses the same technique to deduce a similar integral identity.

To conclude the proof, it is enough to show that the right hand side of~\eqref{eq:Phi-id1} is continuous.  To this end, we first show that the integrand in~\eqref{eq:Phi-id1} stays in $L^1(\R^n\setminus\Omega)$. Using Corollary~\ref{cor:est-f}, we compute
\begin{align*}
\int_{\{f>s\}}\abs*{\frac{\nabla^2 f(\nabla f,\nabla f)}{2(1-f^2)^{\frac{n}{2}}}}\d v_g\,&<\,
C_1\int_{\{f>s\}}\frac{\abs{x}^{-2}}{\abs{x}^{-n}}\cdot\abs{x}^{-2n}\d v
\\
&=\,C_1\int_{\{f>s\}}\abs{x}^{-n-2}\abs{x}^{n-1}\d\sigma_{g_{S^{n-1}}} d \abs{x}
\\
&<\,C_2\int_{\{f>s\}}\abs{x}^{-3}\d \abs{x},
\end{align*}
for some constants $0<C_1<C_2$. Using again the estimates in Corollary~\ref{cor:est-f}, we see that the domain of integration $\{f>s\}=\{1-f<1-s\}$ is contained in $\{\frac{1}{C_3}\abs{x}^{-2}<1-s\}=\{\abs{x}>[C_3(1-s)]^{-\frac{1}{2}}\}$ for some other constant $C_3$. Hence, we get
\begin{equation}
\label{eq:bound-int}
\int_{\{f>s\}}\abs*{\frac{\nabla^2 f(\nabla f,\nabla f)}{2(1-f^2)^{\frac{n}{2}}}}\d v_g\,<\,C_2\int_{[C_3(1-s)]^{-\frac{1}{2}}}^{+\infty}\abs{x}^{-3}\d \abs{x}\,<\,C_4(1-s)\,<\,C_4\,,
\end{equation}
where $C_4$ is, yet again, another constant, possibly bigger than $C_1,C_2,C_3$. This proves that the integrand in~\eqref{eq:Phi-id1} stays in $L^1$.

Now let $\phi^+,\phi^-$ be the positive and negative part of the integrand in~\eqref{eq:Phi-id1}, so that
$$
\frac{\nabla^2 f(\nabla f,\nabla f)}{2(1-f^2)^{\frac{n}{2}}}\,=\,\phi^+-\phi^-\,,
$$
and consider the two measures $\mu^+,\mu^-$ on $\R^n\setminus\Omega$, defined as
$$
\mu^+(E)\,=\,\int_{E}\phi^+\d v_g\,,\qquad \mu^-(E)\,=\,\int_{E}\phi^-\d v_g\,,
$$
for every $v_g$-measurable set $E\subseteq\R^n\setminus\Omega$. Clearly $\mu^+,\mu^-$ are finite (because of the bound~\eqref{eq:bound-int}) positive measures, and they are absolutely continuous with respect to $v_g$. In particular, since the Hausdorff dimension of the level sets of $f$ is bounded by $(n-1)$ (as already noticed in Section~\ref{sec:conformal-form}), we have that the $v_g$-measure of the level sets of $f$ is zero, hence also $\mu^+(\{f=s\})=\mu^-(\{f=s\})=0$ for all $s\in[0,1)$. It follows then from~\cite[Proposition~2.6]{Amb_Dap_Men} that the functions $s\mapsto\mu^+(\{f>s\}), s\mapsto\mu^-(\{f>s\})$ are continuous for all $s\in\coint{0}{1}$. This implies that
$$
\Phi(s)\,=\,-\frac{\mu^+(\{f>s\})-\mu^-(\{f>s\})}{1-s^2}\,.
$$
is continuous, as wished.
\end{proof}

\medskip

\noindent{\bf Differentiability and monotonicity.}
Before addressing the issue of the differentiability of $\Phi$, it is useful to prove the following intermediate result.

\begin{lem}
	\label{le:Psi}
	Let $(\Omega,f,g)$ be a regular solution to problem~\eqref{eq:geom-system2} in the sense of Definition~\ref{ass:fg}, and let $X$ be the vector field defined by~\eqref{eq:def-x2}.
	Then, for every $s \in \coint{0}{1}$, it holds
	\begin{align*} 
	\int_{\{f>s\}}\Div_g X\d v_g
	&=
	-\int_{\{f=s\}} \scal{X}{\nu_g}_g \d \sigma_g 
	\\
	&= \frac{2}{(1-s){(1+s)}^{n-1}} \left[ \int_{\{f=s\}} \abs{\nabla f}_g^2 H_g \d \sigma_g + (n-1) \frac{s}{1-s^2} \int_{\{f=s\}} \abs{\nabla f}_g^3 \d \sigma_g \right] . 
	\end{align*}
\end{lem}

\begin{proof}
	The second equality is just a computation: using~\eqref{eq:def-x2},~\eqref{eq:distr-nabla2},~\eqref{eq:mean-curv}
	and~\eqref{eq:geom-system2} we have
	\begin{align*}
	-\int_{\{f=s\}} \scal{X}{\nu_g}_g \d \sigma_g & = -\int_{\{f=s\}} \frac{1}{\abs{\nabla f}_g} \scal{X}{\nabla f}_g \d \sigma_g \\
	& = -\frac{1}{(1-s){(1+s)}^{n-1}} \int_{\{f=s\}} \left( 2 \frac{\nabla^2 f(\nabla f, \nabla f)}{\abs{\nabla f}_g} - \frac{2}{n} \Delta_g f \abs{\nabla f}_g \right) \d \sigma_g \\
	& = -\frac{2}{(1-s){(1+s)}^{n-1}} \left[ \int_{\{f=s\}} -\abs{\nabla f}^2_g H_g \d \sigma_g + \left(\frac{n-1}{n}\right) \int_{\{f=s\}} \Delta_g f \abs{\nabla f}_g \d \sigma_g \right] \\
	& = \frac{2}{(1-s){(1+s)}^{n-1}} \left[ \int_{\{f=s\}} \abs{\nabla f}^2_g H_g \d \sigma_g + (n-1) \frac{s}{1-s^2} \int_{\{f=s\}} \abs{\nabla f}^3_g \d \sigma_g \right] . 
	\end{align*}
	It remains to prove the identity
	\begin{equation}
	\label{eq:intid-lemma}
	\int_{\{f>s\}}\Div_g X\d v_g
	=
	\int_{\{f=s\}} \scal{X}{\nu_g}_g \d \sigma_g\,. 
	\end{equation}
	Again, as in the proof of Lemma~\ref{le:Phi-cont}, let us start from the case where $s>0$ is a regular value, and let $s<S<1$ be a value big enough so that the level set $\{f=S\}$ is regular. 
	As an application of the divergence theorem, it holds
	\begin{equation}
	\label{eq:intid-div}
	\int_{\{s<f<S\}}\Div_g X\d v_g
	=
	\int_{\{f=S\}} \scal{X}{\nu_g}_g \d \sigma_g-\int_{\{f=s\}} \scal{X}{\nu_g}_g \d \sigma_g\,.
	\end{equation} 
	From Cauchy-Schwarz, we have 
	$$
	\int_{\{f=S\}} \scal{X}{\nu_g}_g \d \sigma_g\,\leq\,\int_{\{f=S\}} \abs{X}_g \d \sigma_g\,,
	$$
	and using Corollary~\ref{cor:asympt} and estimate~\eqref{eq:est-levels-f}, we conclude
	$$
	\lim_{S\to 1^-}\int_{\{f=S\}} \scal{X}{\nu_g}_g \d \sigma_g\,=\,0\,.
	$$
	Therefore, taking the limit of identity~\eqref{eq:intid-div}, we obtain~\eqref{eq:intid-lemma} for all the regular values $0<s<1$.
	
If $s$ is a critical value, one proceeds as in the proof of Lemma~\ref{le:Phi-cont}, that is, one integrates $\Div_g X$ on $\{s<f<S\}\setminus U_\epsilon$, where $U_\epsilon$ is the $\epsilon$-neighborhood of the critical points, and then takes the limit as $\epsilon\to 0$ to recover~\eqref{eq:intid-lemma} (again, we refer to~\cite{agostiniani-mazzieri} for the technical details).
\end{proof} 

\noindent We are finally ready to prove the main result of this section.

\begin{pro}
	\label{pro:almost-main-f}
Let $(\Omega,f,g)$ be a regular solution to problem~\eqref{eq:geom-system2} in the sense of Definition~\ref{ass:fg}. Then the function $\Phi:[0,1)\to \R$ defined by~\eqref{eq:Phi} is differentiable and, for all $s\in\coint{0}{1}$, it holds
\begin{equation}
\label{eq:almost-main-f}
{(1-s^2)}^\frac{n+2}{2} \Phi'(s)= -2 \left[\int_{\{f=s\}} \abs{\nabla f}_g^2 H_g \d \sigma_g +(n-1)\frac{s}{1-s^2} \int_{\{f=s\}} \abs{\nabla f}_g^3 \d \sigma_g \right] \leq 0 \, .
\end{equation}
In fact, the left hand side is increasing in $s$ and it tends to zero as $s \to 1$. Moreover, if $\Phi'(s)=0$ for some $s\in\coint{0}{1}$, then it holds $\nabla^2 f= \lambda g$ on the whole $\R^n\setminus\Omega$, for some $\lambda \in \C^\infty(\R^n \setminus \Omega)$.
\end{pro}

\begin{proof}
 Applying the coarea
 formula to~\eqref{eq:Phi-id1}, we obtain the following:
\begin{equation}
\label{eq:Phi-coarea}
 (1-s^2)\Phi(s)\,=\,-\int_s^1 \left[\int_{\{f=\tau\}} \frac{\nabla^2 f(\nabla f,\nabla f)}{2(1-f^2)^{\frac{n}{2}}} \cdot\frac{1}{\abs{\nabla f}_g}\d \sigma_g \right]\d \tau\,.
\end{equation}
 The statement of the Riemannian coarea formula is in~\cite[Exercise
 III.12, (d)]{chavel}. The missing proof follows from
 Fubini-Tonelli's theorem and~\eqref{eq:coarea}. By the fundamental theorem of calculus, we have that, if the function
 \begin{equation}
 \label{eq:assignment-coarea}
 \tau\mapsto \int_{\{f=\tau\}} \frac{\nabla^2 f(\nabla f,\nabla f)}{2(1-f^2)^{\frac{n}{2}}} \cdot\frac{1}{\abs{\nabla f}_g}\d \sigma_g 
 \end{equation} 
 is continuous, then $\Phi$ is differentiable.
 To prove the continuity of~\eqref{eq:assignment-coarea}, we first compute
 \begin{align}
 \notag
 \int_{\{f=\tau\}} \frac{\nabla^2 f(\nabla f,\nabla f)}{2(1-f^2)^{\frac{n}{2}}} \cdot\frac{1}{\abs{\nabla f}_g}\d \sigma_g\,&=\,
 \int_{\{f=\tau\}} \frac{\scal{\nabla\abs{\nabla f}^2_g}{\nabla f}_g}{(1-f^2)^{\frac{n}{2}}} \cdot\frac{1}{\abs{\nabla f}_g}\d \sigma_g
 \\
 \notag
 &=\, 
 \int_{\{f=\tau\}} \frac{\Scal{(1-f)(1+f)^{n-1}X+\frac{2}{n}\Delta_g f\nabla f}{\nabla f}_g}{(1-f^2)^{\frac{n}{2}}} \cdot\frac{1}{\abs{\nabla f}_g}\d \sigma_g
 \\
 \notag
 &=\,
 \int_{\{f=\tau\}} \left[\left(\frac{1+f}{1-f}\right)^{\!\frac{n-2}{2}}\scal{X}{\nu_g}\,-\,2f\frac{\abs{\nabla f}_g^3}{(1-f^2)^{\frac{n+2}{2}}}\right]\d\sigma_g
 \\
 \label{eq:assignment2-coarea}
 &=\,-\left(\frac{1+\tau}{1-\tau}\right)^{\!\frac{n-2}{2}}\int_{\{f>\tau\}}\Div_g X\d v_g-2\tau\Phi(\tau)\,,
 \end{align} 
 where in the last equality we have used Lemma~\ref{le:Psi}.
 Since we have already proved the continuity of $\Phi$ in Lemma~\ref{le:Phi-cont}, it remains to discuss the continuity of the function
 \begin{equation}
 \tau\mapsto\int_{\{f>\tau\}}\Div_g X\d v_g.
 \end{equation}
 This can be done following the exact same scheme as in the proof of the continuity of $\Phi$ in Lemma~\ref{le:Phi-cont}, namely, one defines a positive finite measure $\mu$ on $\R^n\setminus\Omega$ as
 $$
 \mu(E)\,=\,\int_E\Div_g X\d v_g
 $$
 and shows that $\tau\mapsto \mu(\{f>\tau\})$ is continuous using~\cite[Proposition~2.6]{Amb_Dap_Men}.
 
 We have thus proved the differentiability of $\Phi$. Now, taking the derivative of~\eqref{eq:Phi-coarea}, and using~\eqref{eq:assignment2-coarea}, we obtain
\begin{equation}
\label{eq:Phi'}
 (1-s^2)\Phi'(s)-\cancel{2s\Phi(s)}\,=\,-\left(\frac{1+s}{1-s}\right)^{\!\frac{n-2}{2}}\int_{\{f>s\}}\Div_g X\d v_g-\cancel{2s\Phi(s)}\,.
\end{equation}
 Identity~\eqref{eq:almost-main-f} follows from~\eqref{eq:Phi'} and Lemma~\ref{le:Psi}. The monotonicity of $\Phi$ is a consequence of the nonnegativity of the divergence of $X$, proved in Lemma~\ref{le:div_X}.
 
 Suppose now $\Phi'(s) = 0$ for some $s \in \coint{0}{1}$.
 Comparing~\eqref{eq:almost-main-f} and Lemma~\ref{le:Psi} we have that
 $\Div_g X = 0$ on $\{f>s\}$. Since $g$ and $f$ are analytic, so is $\Div_g X $, hence, by unique continuation, it holds $\Div_g X = 0$ on the whole $\R^n\setminus\Omega$. The thesis then follows from
 Lemma~\ref{le:div_X}.
\end{proof}

\noindent With Proposition~\ref{pro:almost-main-f}, we have almost completed the proof of Theorem~\ref{theo:main-thm-f} stated in the introduction. It only remains to show that the condition $\nabla^2 f= \lambda g$, $\lambda \in \C^\infty(\R^n \setminus \Omega)$, implies the spherical symmetry of the solution. This will be the object of Section~\ref{sec:equality}.

\medskip

\noindent{\bf Theorem~\ref{theo:main-thm-f} implies Theorem~\ref{theo:main-thm}.}
To conclude the section, we show that inequality~\eqref{eq:almost-main-f}, proved above, is equivalent to inequality~\eqref{eq:main-thm} in Theorem~\ref{theo:main-thm} (as pointed out above, the rigidity statements will be discussed later, in Section~\ref{sec:equality}). This is a rather straightforward
computation, using the formul\ae\  developed in
Section~\ref{sec:conformal-form}.
\begin{align*}
  \frac{\nabla^2_{ij} f \nabla_i f \nabla_j f}{\abs{\nabla f}_g^2} & = \frac{{(1-f)}^{-4} \nabla^2_{ij} f D_i f D_j f}{{(1-f)}^{-2} \abs{Df}^2} \\
  & = \frac{1}{{(1-f)}^2} \left[ \frac{n}{1-f^2} \abs{Df}^2 - \frac{1+f}{1-f^2} \abs{Df}^2 - \frac{1}{n-2} \left( \frac{1+f}{1-f} \right)^{\frac{n-2}{2}} (1-f^2) D^2_{ij} u \frac{D_i f D_j f}{\abs{Df}^2} \right] ,
\\
  \abs{\nabla f}_g H_g & = \Delta_g f - \frac{\nabla^2_{ij} f \nabla_i f \nabla_j f}{\abs{\nabla f}_g^2} \\
  & = (-nf) \frac{{(1-f)}^{-2} \abs{Df}^2}{1-f^2} - \frac{\nabla^2_{ij} f \nabla_i f \nabla_j f}{\abs{\nabla f}_g^2} \\
  & = - (n-1) \frac{\abs{Df}^2}{{(1-f)}^3} + \frac{1}{n-2} \left( \frac{1+f}{1-f} \right)^{\frac{n}{2}} D^2_{ij} u \frac{D_i f D_j f}{\abs{Df}^2} \\
  & = \frac{n-1}{n-2} \left[ - \frac{1}{n-2} \left( \frac{1+f}{1-f} \right)^{n-1} (1+f) \abs{Du}^2 + \frac{1}{n-1} \left( \frac{1+f}{1-f} \right)^{\frac{n}{2}} D^2_{ij} u \frac{D_i u  D_j u}{\abs{Du}^2} \right] ,
\\
  \d \sigma_g & = {(1-f)}^{n-1} \d \sigma .
\end{align*}
We are now able to rewrite the summands of~\eqref{eq:main-thm} in
terms of $u$ on a level set $\{f=s\}=\{u=t\}$. The first one is:
\begin{align*}
& -2 \int_{\{f=s\}} \abs{\nabla f}^2_g H_g \d \sigma_g \\
& \qquad {} = -2 \int_{\{f=s\}} {(1-f)}^{n-1} {(1-f)}^{-1} \abs{Df} \cdot \abs{\nabla f}_g H_g \d \sigma \\
& \qquad {} = -\frac{2}{n-2} {(1-s){(1+s)}^{n-1}} \int_{\{f=s\}} \left( \frac{1+f}{1-f} \right)^{-(n-1)+\frac{n-2}{2}+1} \abs{Du} \cdot \abs{\nabla f}_g H_g \d \sigma \\
& \qquad {} = -\frac{2(n-1)}{{(n-2)}^2} {(1-s){(1+s)}^{n-1}} \int_{\{f=s\}} \abs{Du} \cdot \left[ \frac{1}{n-1} \frac{D^2_{ij} u D_i u D_j u}{\abs{Du}^2} \left( \frac{1+f}{1-f} \right) - \frac{\abs{Du}^2}{n-2} \left( \frac{1+f}{1-f} \right)^{\frac{n}{2}} (1+f) \right] \d \sigma \\
& \qquad {} = -\frac{2(n-1)}{{(n-2)}^2} {(1-s){(1+s)}^{n-1}} \int_{\{u=t\}} \abs{Du}^2 \cdot \left[ \frac{H}{n-1} \left( \frac{1+f}{1-f} \right) - \frac{\abs{Du}}{n-2} \left( \frac{1+f}{1-f} \right)^{\frac{n}{2}} (1+f) \right ] \d \sigma .
\end{align*}
The second summand is:
\begin{align*}
& -2 \cdot \frac{(n-1)s}{1-s^2} \int_{\{f=s\}} \abs{\nabla f}_g^3 \d \sigma_g \\
& \qquad {} = -\frac{2(n-1)}{{(n-2)}^2} \int_{\{f=s\}} {(1-f)}^{n-1} \cdot \frac{f}{1-f^2} \cdot {(1-f)}^{-3} \left( \frac{1+f}{1-f} \right)^{3\frac{n-2}{2}} {(1-f^2)}^3 \abs{Du}^3 \d \sigma \\
& \qquad {} = -\frac{2(n-1)}{{(n-2)}^2}{(1-s){(1+s)}^{n-1}} \int_{\{f=s\}} \left( \frac{1+f}{1-f} \right)^{-(n-1)+3\frac{n-2}{2}+2} \cdot f \cdot \abs{Du}^3 \d \sigma \\
& \qquad {} = -\frac{2(n-1)}{{(n-2)}^2}{(1-s){(1+s)}^{n-1}} \int_{\{u=t\}} \left( \frac{1+f}{1-f} \right)^{\frac{n}{2}} \cdot f \cdot \abs{Du}^3 \d \sigma .
\end{align*}
From~\eqref{eq:def-f} we see that
$\frac{1+f}{1-f} = u^{-\frac{2}{n-2}}$, therefore:
\begin{align*}
& -2 \int_{\{f=s\}} \abs{\nabla f}^2_g H_g \d \sigma_g -2 \cdot \frac{(n-1)s}{1-s^2} \int_{\{f=s\}} \abs{\nabla f}_g^3 \d \sigma_g \\
& \qquad = -\frac{2(n-1)}{{(n-2)}^2} {(1-s){(1+s)}^{n-1}} \int_{\{u=t\}} \abs{Du}^2 \cdot \left(  \frac{H}{n-1} u^{-\frac{2}{n-2}} - \frac{\abs{Du}}{n-2} u^{-\frac{n}{n-2}}  \right) \d \sigma ,
\end{align*}
which is (up to a positive factor) the quantity that appears in
\eqref{eq:main-thm}.

\section{Geometric inequalities via strong maximum principle}
\label{sec:proof-part2}

The aim of this section is to prove inequality~\eqref{eq:main-thm2-f} in Theorem~\ref{theo:main-thm2-f} and to show how this is used to recover~\eqref{eq:main-thm2} in Theorem~\ref{theo:main-thm2}.

\medskip

\noindent{\bf Definition of $\gamma$ and regularity of ${\rm R}_\gamma$.}
We start by recalling some notations that were anticipated in Section~\ref{sec:strategy}. Given a regular solution $(\Omega,u)$ of problem~\eqref{eq:system}, we set $\rho={[\Ca(\Omega)]}^{{1}/{(n-2)}}$. We have already noticed that, in the case where $\Omega$ is a ball, then $\rho$ is equal to its radius. We denote by $S^n_\rho$ the sphere of radius $\rho$ with its standard embedding in $\R^{n+1}$, and we define the stereographic projection
\begin{equation*}
  \pi_\rho \colon S_\rho^n \setminus \{N\}  \longrightarrow \R^n\,, \qquad
  (z', z^{n+1})  \longmapsto \frac{\rho\,z'}{\rho-z^{n+1}} \, .
\end{equation*}
Let $f$ and $g$ be defined by~\eqref{eq:def-f} and~\eqref{eq:def-g}, respectively. As anticipated in the introduction, we consider the compact space
\begin{equation*}
M := \pi_\rho^{-1}(\R^n \setminus \Omega) \cup \{N\} \subset S_\rho^n
\end{equation*}
endowed with the metric 
\begin{equation}
\label{eq:def-h}
\def\arraystretch{1.5}
\gamma_{|_{y}}\,=\,
\left\{ \begin{array}{ll}
{(\pi_\rho^* g)}_{|_{y}} & \mbox{if }y\in\pi_\rho^{-1}(\R^n \setminus\Omega)\,,
\\
{(g_{S_\rho^n})}_{|_N} & \mbox{if }y=N\,.
\end{array}\right.
\end{equation}
We observe that the metric $\gamma$ can be seen as a conformal modification of the standard metric $g_{S_\rho^n}$. In fact, recalling the definition of the spherical metric $g_{\text{sph}} = {(1-f_0)}^2 g_{\R^n}$ on $\R^n$, where $f_0$ is as in~\eqref{eq:f0}, we compute 
\begin{equation*}
\pi_\rho^* g = {(1-f \circ \pi_\rho)}^2 \pi_\rho^* g_{\R^n} = \left( \frac{1-f
	\circ \pi_\rho}{1-f_0 \circ \pi_\rho} \right)^2 \pi^*_\rho g_{\rm sph}
= \left( \frac{1-f
	\circ \pi_\rho}{1-f_0 \circ \pi_\rho} \right)^2 {[\Ca(\Omega)]}^{-\frac{2}{n-2}} g_{S_{\rho}^n} \,.
\end{equation*}
In particular, using Corollary~\ref{cor:est-f} we see that the conformal factor in the rightmost hand side tends to $1$ at the north pole $N$, so that the metric $\gamma$ 
is in fact continuous at $N$ (and of course it is smooth on $M\setminus \{N\}$).
In order to show that $\gamma$ is also
$\C^2$ in a neighborhood of $N$, it is sufficient to prove that the conformal factor has the same regularity. Again, this readily follows from Corollary~\ref{cor:est-f}. 
Since $\gamma$ is $\C^2$, its scalar curvature ${\rm R}_\gamma$ is well-defined and continuous on the whole $M$.
Moreover, ${\rm R}_\gamma$ is clearly $\C^\infty$ in
$M\setminus\{N\}$. The following lemma tells us more about the
regularity of ${\rm R}_\gamma$ at the north pole.

\begin{lem}
	\label{le:w-reg}
Let $(\Omega,f,g)$ be a regular solution to problem~\eqref{eq:geom-system2} in the sense of Definition~\ref{ass:fg}. Consider the compact domain $M = \pi_\rho^{-1}(\R^n \setminus \Omega) \cup \{N\} \subset S^n_\rho$, equipped with the metric $\gamma$ defined in~\eqref{eq:def-h}. Then the scalar curvature ${\rm R}_\gamma$ of the metric $\gamma$ belongs to $\mathcal{W}^{1,2}(M)\cap\C^0(M)$.
\end{lem}

\begin{proof}
Since ${\rm R}_\gamma$ is continuous and $M$ is compact, clearly ${\rm R}_\gamma\in L^2(M)$.
 We want to show that also its gradient $\nabla {\rm R}_\gamma$ belongs to $ L^2(M)$. Since ${\rm R}_\gamma$ is smooth on $M\setminus\{N\}$ and $M$ is compact, it is enough to prove that the function
  $\nabla {\rm R}_\gamma$ belongs to $L^2$ near the north pole.
 We observe that, from~\eqref{eq:R} and~\eqref{eq:def-x2}, we have 
  $$
  \nabla {\rm R}_\gamma(y)=n(n-1){(1+f)}^{n-2} X(\pi_\rho(y))
  $$ 
  for all $y\in M\setminus\{N\}$. In particular, it holds
  $\abs{\nabla {\rm R}_\gamma}_\gamma^2 (y) = n^2{(n-1)}^2{{(1+f)}^{2(n-2)}} \abs{X}_g^2(\pi_\rho(y))$, and using Corollary~\ref{cor:asympt}:
  \[ \abs{\nabla {\rm R}_\gamma}_\gamma^2 (y)= \abs{\pi_\rho(y)}^2 \cdot o(1) , \]
  as $y\to N$ (or equivalently, as $\abs{\pi_\rho(y)}\to\infty$).
Thus, for all $T>0$, we obtain:
  \begin{align*}
    \int_{M\cap\{\abs{\pi_\rho(y)}>T\}} \abs{\nabla {\rm R}_\gamma}_\gamma^2\d v_\gamma & = \int_{\R^n\cap\{\abs{x}>T\}} \abs{x}^2 o(1)
  \cdot \abs{x}^{-2n} \d v 
  \\
  &=  o(1)\cdot\int_{\R^n\cap\{\abs{x}>T\}} \abs{x}^{-2n+2} \abs{x}^{n-1}
  \d \sigma_{g_{S^{n-1}}} d \abs{x} \\
& = o(1) \cdot
  \int_{T}^{\infty} \abs{x}^{-n+1} \d \abs{x}
  \\
  & = o(1) \cdot
  T^{2-n} .
  \end{align*}
  Since $n \geq 3$, this last term goes to zero as $T$ goes to infinity.

Now it only remains to prove that $\nabla {\rm R}_\gamma$ is the weak derivative of ${\rm R}_\gamma$, that is, we want to show 

\begin{equation}
\label{eq:weak-der}
\int_{M}\psi\nabla_i {\rm R}_\gamma\d v_\gamma\,=\,-\int_{M}{\rm R}_\gamma\nabla_i\psi \d v_\gamma ,\qquad \text{for all $\psi\in\C_c^\infty(M),\ i=1,\dots,n$.}
\end{equation}  
  Fix $\psi\in \C_c^\infty(M)$. For all $T>0$, we can write $\psi=\alpha_T+\beta_T$, where $\alpha_T,\beta_T$ are functions with support contained in $\set{\abs{\pi_\rho(y)}>T}\ni N$ and $\set{\abs{\pi_\rho(y)}<2T}$, respectively.
It is important to notice that, from the estimates on $g$ in Corollary~\ref{cor:asympt}, it easily follows that the distance between the sets $\R^n\cap\{\abs{x}<T\}$ and $\R^n\cap\{\abs{x}>2T\}$ is $O(T^{-1})$ as $T\to +\infty$. Pulling back on the sphere, we have that also the distance between $M\cap\set{\abs{\pi_\rho(y)}<T}$ and $M\cap\set{\abs{\pi_\rho(y)}>2T}$ is $O(T^{-1})$ as $T\to +\infty$.
From this follows that we can choose the function $\alpha_T$ in such a way that it holds
$$
  |\nabla\alpha_T|_\gamma\leq CT\,
  $$
  on the whole $M$, for some constant $C>0$.

  Notice that the support of $\beta_T$ does not contain the north pole, hence, from the smoothness of ${\rm R}_\gamma$ in $M\setminus\{N\}$, it immediately follows
  $$
  \int_{M}\beta_T\nabla_i {\rm R}_\gamma\d v_\gamma\,=\,-\int_{M}{\rm R}_\gamma\nabla_i\beta_T \d v_\gamma.
  $$
  Therefore, in order to prove~\eqref{eq:weak-der}, it is enough to show that both $\int_{M}\alpha_T\nabla_i {\rm R}_\gamma\d v_\gamma$ and $\int_{M}{\rm R}_\gamma\nabla_i\alpha_T \d v_\gamma$ go to zero as $T\to +\infty$.
  Using again $\abs{\nabla {\rm R}_\gamma}_\gamma^2(y) = \abs{\pi_\rho(y)}^2 \cdot o(1)$, we obtain, for all $T>0$:
  \begin{align*}
  \abs*{\int_{M}\alpha_T\nabla_i {\rm R}_\gamma\d v_\gamma}\,&\leq\,\max_M\abs{\alpha_T}\cdot\int_{M\cap\{\abs{\pi_\rho(y)}>T\}}\!\!|\nabla {\rm R}_\gamma|_\gamma\d v_\gamma
  \\
  &=\,\int_{M\cap\{\abs{\pi_\rho(y)}>T\}}\!\!|\pi_\rho(y)| \cdot o(1)\d v_\gamma
  \\
  &=\,\int_{\R^n\cap\{\abs{x}>T\}}\!\!|x| \cdot o(1)\d v_g
  \end{align*}
  where in the last equality we have pulled back the integral on $\R^n$. Recalling the relation between the densities $dv$ and $dv_g$ shown in Corollary~\ref{cor:asympt}, we compute
  \begin{align*}
  \abs*{\int_{M}\alpha_T\nabla_i {\rm R}_\gamma\d v_\gamma}\,&\leq\,\int_{\R^n\cap\{\abs{x}>T\}}\!\!|x|^{-2n+1} \cdot o(1)\d v
  \\
  &=\,\int_{\R^n\cap\{\abs{x}>T\}}\!\!\abs{x}^{-2n+1} \cdot o(1)\cdot \abs{x}^{n-1}\d\sigma_{g_{S^{n-1}}} d \abs{x}
  \\
  &=\,o(1)\cdot\int_T^{+\infty} \abs{x}^{-n}d\abs{x}
  \\
  &=\,o(1)\cdot T^{1-n}\,,
  \end{align*}
  hence this integral goes to zero as $T\to +\infty$. 
  
  The estimate of the second integral is done in a similar fashion. Recalling $\abs{\nabla \alpha_T}_\gamma\leq CT$ for some constant $C$, and using again Corollary~\ref{cor:asympt}:
  \begin{align*}
  \abs*{\int_{M}{\rm R}_\gamma\nabla\alpha_T\d v_\gamma}\,&\leq\,\max_M({\rm R}_\gamma)\cdot\int_{M\cap\{\abs{\pi_\rho(y)}>T\}}\!\!|\nabla \alpha_T|_\gamma\d v_\gamma
  \\
  &\leq\,C_1T\int_{\R^n\cap\{\abs{x}>T\}}\!\!d v_g
  \\
  &\leq\,C_2T\int_{\R^n\cap\{\abs{x}>T\}}\abs{x}^{-2n}\d v
  \\
  &=\,C_2T\int_{\R^n\cap\{\abs{x}>T\}}\!\!\abs{x}^{-2n} \cdot \abs{x}^{n-1}\d\sigma_{g_{S^{n-1}}}d \abs{x}
  \\
  &\leq\,C_3 T^{1-n}\,,
  \end{align*}
  where $0<C_1<C_2<C_3$ are constants and we have used the fact that ${\rm R}_\gamma$ is bounded on $M$ (because ${\rm R}_\gamma$ is continuous and $M$ is compact).
  Therefore also this goes to zero as $T\to+\infty$.
  This proves that $\nabla {\rm R}_\gamma$ is indeed the weak derivative of ${\rm R}_\gamma$ in $M$.
\end{proof}

\medskip

\noindent{\bf Elliptic inequality for ${\rm R}_\gamma$ and maximum principle. } Pulling back the quantities in formula~\eqref{eq:ellin-R2}, we find that the scalar curvature ${\rm R}_\gamma$ satisfies the following second order elliptic inequality
\begin{equation}
\label{eq:lapl_v}
\Delta_\gamma {\rm R}_\gamma \,-\,(n-2)\Scal{\frac{\nabla \phi}{1+\phi}}{\nabla {\rm R}_\gamma}_{\!\gamma}\,\,\geq\,\, 0
\end{equation}
pointwise on $M\setminus \{N\}$, where we have denoted by $\phi$ the completion of the pull-back of $f$. Namely, $\phi:M\to\R$ is defined by
$$
\def\arraystretch{1.5}
\phi(y)\,=\,
\left\{ \begin{array}{ll}
f\circ\pi_\rho(y) & \mbox{if }y\in\pi_\rho^{-1}(\R^n \setminus\Omega)\,,
\\
1 & \mbox{if }y=N\,.
\end{array}\right.
$$
Notice that $\phi$ is smooth in $M\setminus \{N\}$, and it is also $\mathcal{C}^2$ near $N$, as it can be easily deduced from Corollary~\ref{cor:est-f}. Having this in mind, we are going to prove that the differential inequality~\eqref{eq:lapl_v} holds in a weak sense on the whole $M$. Combining this with a weak version of the strong maximum principle, we deduce the following theorem.
%
%
%
%
%
%

\begin{pro}
	\label{pro:max_pr}
		Let $(\Omega,f,g)$ be a regular solution to problem~\eqref{eq:geom-system2} in the sense of Definition~\ref{ass:fg}. Consider the domain $M = \pi_\rho^{-1}(\R^n \setminus \Omega) \cup \{N\} \subset S^n_\rho$, equipped with the metric $\gamma$ defined by~\eqref{eq:def-h}. Then, for any $y\in M$, it holds 
	\begin{equation}
	\label{eq:max_pr}
	{\rm R}_\gamma(y)\leq \max_{\de M}{\rm R}_\gamma\,.
	\end{equation}
	Moreover, the equality is fulfilled for some $y\in M\setminus\de M=\pi_\rho^{-1}(\R^n\setminus\overline{\Omega})\cup\{N\}$, if and only if $\nabla^2 f= \lambda g$ on the whole $\R^n\setminus\Omega$, for
	some $\lambda \in \C^\infty(\R^n \setminus \Omega)$.
\end{pro}

\begin{proof}
The strategy of the proof consists in showing that inequality~\eqref{eq:lapl_v}
holds in the weak sense on the whole $M$, so that we can use a strong maximum principle (see for example~\cite[Theorem~8.19]{gilbarg-trudinger}) to deduce the thesis.
This amounts to prove that for every nonnegative $\psi\in\C_c^\infty(M)$ it holds
\begin{equation}
\label{eq:thesis}
\int_{M}\left[\scal{\nabla \psi}{\nabla {\rm R}_\gamma}_\gamma+(n-2)\psi\Scal{\frac{\nabla \phi}{1+\phi}}{\nabla {\rm R}_\gamma}_{\!\gamma}\right]\d v_\gamma \,\, \leq \,\,0 \,.
\end{equation}
Fix a nonnegative $\psi\in \C_c^\infty(M)$. As in the proof of Lemma~\ref{le:w-reg}, for all $T>0$ we write $\psi=\alpha_T+\beta_T$, where $\alpha_T,\beta_T$ are nonnegative functions with support contained in $\set{\abs{\pi_\rho(y)}>T}\ni N$ and $\set{\abs{\pi_\rho(y)}<2T}$, respectively. As already noticed in the proof of Lemma~\ref{le:w-reg}, the function $\alpha_T$ can be chosen in such a way that $
|\nabla\alpha_T|_\gamma\leq CT$, for some constant $C>0$. Furthermore, it holds $\abs{\nabla {\rm R}_\gamma}_\gamma^2 (y)= \abs{\pi_\rho(y)}^2 \cdot o(1)$ as $y$ approaches the north pole. We also observe that $\abs{\nabla\phi}_\gamma=\abs{\pi_\rho(y)}\cdot(1+o(1))$, as it follows from Corollary~\ref{cor:asympt}.

Since the support of $\beta_T$ does not contains the north pole, and ${\rm R}_\gamma$ is smooth in $M\setminus \{N\}$, we immediately have the following estimate
\begin{multline}
\label{eq:weak_v_1}
\int_{M}\left[\scal{\nabla \beta_T}{\nabla {\rm R}_\gamma}_\gamma+(n-2)\beta_T\Scal{\frac{\nabla \phi}{1+\phi}}{\nabla {\rm R}_\gamma}_{\!\gamma}\right]\d v_\gamma
\,=
\\
=\,-\int_{M\cap\set{\abs{\pi_\rho(y)}<2T}}\beta_T\left[\Delta_\gamma {\rm R}_\gamma \,-\,(n-2)\Scal{\frac{\nabla \phi}{1+\phi}}{\nabla {\rm R}_\gamma}_{\!\gamma}\right]\d v_\gamma\leq 0\,,
\end{multline}
where in the last inequality we have used formula~\eqref{eq:lapl_v}, that holds outside of $N$.

The right-hand side of~\eqref{eq:weak_v_1} being independent of $T$, in order to prove~\eqref{eq:thesis} it is enough to show that
\begin{equation}
\label{eq:thesis-psi1}
\lim_{T\to+\infty}\int_{M}\left[\scal{\nabla \alpha_T}{\nabla {\rm R}_\gamma}_\gamma+(n-2)\alpha_T\Scal{\frac{\nabla \phi}{1+\phi}}{\nabla {\rm R}_\gamma}_{\!\gamma}\right]\d v_\gamma\,=\,0.
\end{equation}
 Converting as usual to
$\R^n$ with the metric $g$, from
Corollary~\ref{cor:est-f} and Corollary~\ref{cor:asympt} we obtain:
\begin{align}
 &\abs*{\int_{M \cap \{\abs{\pi_\rho(y)}>T\}}\left[\scal{\nabla \alpha_T}{\nabla {\rm R}_\gamma}_\gamma+(n-2) \alpha_T \Scal{\frac{\nabla \phi}{1+\phi}}{\nabla {\rm R}_\gamma}_{\!\gamma}\right]\d v_\gamma}  \leq 
 \nonumber
 \\
  &	\quad\qquad\qquad\qquad\qquad\qquad\qquad \leq \int_{M \cap \{\abs{\pi_\rho(y)}>T\}} \left[\abs{\nabla  \alpha_T}_\gamma +(n-2)\abs{\alpha_T}\cdot\abs*{\frac{\nabla\phi}{1+\phi}} \right] \cdot\abs{\nabla {\rm R}_\gamma}_\gamma\d v_\gamma \nonumber
  \\
  & \quad\qquad\qquad\qquad\qquad\qquad\qquad = \int_{M \cap \{\abs{\pi_\rho(y)}>T\}} \Big(T\abs{\pi_\rho(y)} + \abs{\pi_\rho(y)}^2 \Big)\cdot o(1)\d v_\gamma \nonumber \\
   & \quad\qquad\qquad\qquad\qquad\qquad\qquad = \int_{M \cap \{\abs{\pi_\rho(y)}>T\}}  \abs{\pi_\rho(y)}^2 \cdot o(1)\d v_\gamma \nonumber \\
  & \quad\qquad\qquad\qquad\qquad\qquad\qquad = \int_{\R^n\cap\{\abs{x}>T\}} \abs{x}^2 \cdot o(1) \cdot \abs{x}^{-2n} \d v \nonumber \\
  & \quad\qquad\qquad\qquad\qquad\qquad\qquad = o(1) \cdot \int_{\R^n\cap\{\abs{x}>T\}} \abs{x}^{-2n+2} \abs{x}^{n-1} \d \sigma_{g_{S^{n-1}}} d \abs{x} \nonumber \\
  & \quad\qquad\qquad\qquad\qquad\qquad\qquad = T^{2-n} \cdot o(1) . \label{eq:weak_v_2}
\end{align}
Inequalities~\eqref{eq:weak_v_1} and~\eqref{eq:weak_v_2} are true for any $T>0$ and $n\geq 3$, hence~\eqref{eq:thesis} is proved and ${\rm R}_\gamma$ indeed solves~\eqref{eq:ellin-R2} in the $\mathcal{W}^{1,2}(M)$--sense. 

We are now in the position to apply
~\cite[Theorem~8.19]{gilbarg-trudinger}. The version that we need of this result states that, if a continuous $\mathcal{W}^{1,2}$ function satisfies an elliptic inequality of the second order  in the $\mathcal{W}^{1,2}$--sense (with some standard hypoteses on the coefficients, that are trivially satisfied in our case), then the strong maximum principle applies.
We have already proved that the function ${\rm R}_\gamma$ belongs to $\C^0(M)\cap\mathcal{W}^{1,2}(M)$ and that it satisfies the elliptic inequality~\eqref{eq:lapl_v} in the weak sense, hence we have ${\rm R}_\gamma\leq \max_{\de M} {\rm R}_\gamma$ as wished.
 Moreover, if equality holds in~\eqref{eq:max_pr}, then ${\rm R}_\gamma$ is constant on
 $M$. Therefore, from~\eqref{eq:R} and~\eqref{eq:def-x2} we deduce that the field $X$ is everywhere
 vanishing, and so is its divergence. We can now conclude using Lemma~\ref{le:div_X}.
\end{proof}

\noindent In order to conclude the proof of Theorem~\ref{theo:main-thm2-f}, it only remains to show that the identity $\nabla^2 f = \lambda g$ forces $\Omega$ and $f$ to be rotationally symmetric (the fact that $(M,\gamma)$ is isometric to a sphere follows easily). This will be the subject of Section~\ref{sec:equality}.

\medskip

\noindent{\bf Theorem~\ref{theo:main-thm2-f} implies Theorem~\ref{theo:main-thm-ext2}. } We conclude this section showing that inequality~\eqref{eq:max_pr} implies formul\ae~\eqref{eq:theo-ext-1}. Evaluating~\eqref{eq:max_pr} at $y=N$ and recalling from~\eqref{eq:R},~\eqref{eq:est_f2} that
 \begin{equation}
 \label{eq:aux2}
 {\rm R}_\gamma=n(n-1)\frac{\abs{\nabla\phi}^2}{1-\phi^2}\,, \quad\hbox{ and } \quad\frac{\abs{\nabla\phi}^2}{1-\phi^2}(N)={[\Ca(\Omega)]}^{-\frac{2}{n-2}}	\,,
\end{equation}
 we obtain
 \begin{equation}
 \label{eq:aux1}
 \left[\frac{1}{\Ca(\Omega)}\right]^{\frac{2}{n-2}}\leq \max_{\de M}\frac{\abs{\nabla \phi}_\gamma^2}{1-\phi^2}\,.
 \end{equation}
Since $\phi=f\circ\pi_\rho$, $\de M=\pi_\rho^{-1}(\de\Omega)$, $f=0$ on $\de\Omega$, $u=1$ on $\de\Omega$, using formula~\eqref{eq:|naf|} we find
\begin{equation}
\label{eq:aux}
\max_{\de M}\frac{\abs{\nabla \phi}_\gamma^2}{1-\phi^2}\leq \max_{\de \Omega}\abs{\nabla f}^2_g\,=\, \max_{\de \Omega}\abs*{\frac{D u}{n-2}}^2\,.
\end{equation}
Comparing formul\ae~\eqref{eq:aux1},~\eqref{eq:aux}, we obtain the first inequality in~\eqref{eq:theo-ext-1}. In order to prove the second inequality in~\eqref{eq:theo-ext-1}, we evaluate~\eqref{eq:max_pr} at a point $y\in\pi_\rho^{-1}(\R^n\setminus\Omega)$. Recalling~\eqref{eq:aux} and~\eqref{eq:aux2}, we have
$$
 \frac{\abs{\nabla \phi}_\gamma^2}{1-\phi^2}(y)\,\leq\, \max_{\de \Omega}\abs*{\frac{D u}{n-2}}^2\,.
 $$ 
Since $\phi=f\circ\pi_\rho$, calling $x=\pi_\rho(y)$ and using~\eqref{eq:|naf|}, we obtain
$$
u^{-2\left(\frac{n-1}{n-2}\right)}\,\abs*{\frac{D u}{n-2}}^2(x)\,\,\leq\,\, \max_{\de \Omega}\abs*{\frac{D u}{n-2}}^2\,,
$$
for all $x\in\R^n\setminus\overline{\Omega}$.
An easy algebraic manipulation of this inequality leads to the second formula in~\eqref{eq:theo-ext-1}.

%
%
%
%

\section{Proof of the rigidity statements}
\label{sec:equality}

In the two previous sections we obtained the equation
$\nabla^2 f = \lambda g$, $\lambda\in\C^{\infty}(\R^n\setminus\Omega)$, as a byproduct of the equality case. Such
condition is very strong, and in this section we show that it induces a very rigid geometric structure
on the problem. 

In order to simplify the formul\ae, {\em throughout this section we will assume
$\Ca(\Omega)=1$}.
The case of a general capacity can be addressed following the exact same strategy, at the cost of slightly more cumbersome computations. Alternatively, one can always reduce himself to study the case $\Ca(\Omega)=1$ via an homothety of $\Omega$, and an appropriate rescaling of $u$ so that the boundary condition in~\eqref{eq:system} remains satisfied. It is easy to check that the new $f,g$, obtained from the rescaled $u$ via~\eqref{eq:def-f},~\eqref{eq:def-g}, still satisfy $\nabla^2 f = \lambda g$, for some (possibly different) function $\lambda$.

\begin{pro}
  \label{pro:warped}
  	Let $(\Omega,f,g)$ be a regular solution to problem~\eqref{eq:geom-system2} in the sense of Definition~\ref{ass:fg}, with $\Ca(\Omega)=1$. Suppose that the
  equation $\nabla^2 f = \lambda g$ is satisfied on
  $\R^n \setminus \Omega$, where
  $\lambda \in \C^\infty(\R^n \setminus \Omega)$. Then:
  \begin{itemize}
  \item The function $f$ is regular everywhere and in particular
    $\abs{\nabla f}_g^2 = 1-f^2 \neq 0$ on $\R^n \setminus \Omega$.
  \item The region $\R^n \setminus \Omega$ is isometric to
    $\coint{0}{\frac{\pi}{2}} \times \de\Omega$ with the warped product metric
    $$
    dr \otimes dr + \cos^2 (r) \cdot g_{\alpha\beta}^{\de\Omega}\cdot dx^\alpha \otimes dx^\beta,
    $$ where
    $g_{\alpha\beta}^{\de\Omega}$ is the restriction of the metric $g$ to $\de\Omega$.
  \item The above isometry sends the level sets of $f$ to level sets of $r$;
    in particular, it sends $\{f=s\}$ to $\{r=\arcsin s\}$ for all
    $s \in \coint{0}{1}$. This implies that a point tends
    to $\infty$ in $\R^n$ if and only if its image have
    $r \to \frac{\pi}{2}$.
  \end{itemize}
\end{pro}

\begin{proof}
Let $s \in \coint{0}{1}$ be a regular value for $f$ and take a vector field $Y$ tangent to
$\{f=s\}$ at a point $p$. Then
\[ \nabla_Y \abs{\nabla f}_g^2 = 2 \nabla^2 f(Y, \nabla f) = 2 \lambda
\scal{Y}{\nabla f}_g = 0 , \]
so $\abs{\nabla f}_g$ is locally constant on every regular level set of $f$. In
particular it can be locally written as a function of $f$.  Let us
define $S$ to be the supremum of all $s \in \coint{0}{1}$ such that
$\abs{\nabla f}_g \neq 0$ on the region $\{f < s \}$.  By Hopf's
lemma (see~\cite[Lemma 10.55]{helms}), $0$ is a regular value of $f$ and since the set of regular
values is always open this implies that $S > 0$. All the computations from now on will be done in the region $\{ f < S \}$.
By standard Morse theory (see for example
\cite[Theorem 2.2]{hirsch}) the region $\{f<S\}$ is diffeomorphic to $\coint{0}{1} \times \{f=0\} = \coint{0}{1}
\times \de\Omega $, so local coordinates
on $\de\Omega$ can be extended, together with $f$, on the whole region $\{f<S\}$.

On the region $\{f<S\}$, the function $\abs{\nabla f}_g(f)$ is well-defined and it
does not vanish, so we can use it to define the function $r(f)$ given
by the ordinary differential equation:
\[ 
\def\arraystretch{1.8}
\left\{ \begin{array}{ll}
  \frac{dr}{df} = \frac{1}{\abs{\nabla f}_g} & \text{for $0 \leq f < S$}, \\
  r(0) = 0 .
\end{array} \right. \]
It is easy to see that $r$ can be used as a local coordinate instead of $f$ and that the metric can be
written as
\[ g = dr \otimes dr + g_{\alpha\beta}(r, x) dx^\alpha \otimes dx^\beta . \]
From now on we use Greek letters for indices that must range over tangential directions.
We denote by $f'$ and $f''$ the derivatives of $f$ with
respect to the new variable $r$. From the definition of $r$ we have that $f' = \frac{df}{dr} = \abs{\nabla f}_g$.
The Christoffel symbols of $g$ have the following form:
\[ \Gamma_{rr}^r = \Gamma_{rr}^\gamma = \Gamma_{\alpha r}^r = 0 , \qquad
\Gamma_{\alpha\beta}^r = - \frac{1}{2} \de_r g_{\alpha\beta} , \qquad
\Gamma_{\alpha r}^\gamma = \frac{1}{2} g^{\gamma\delta} \de_r g_{\alpha\delta} .
\]
Computing the Hessian with these coordinates:
\[ \nabla^2 f = f'' dr \otimes dr + f' \nabla^2 r = f'' dr \otimes dr + \frac{1}{2} f' \de_r g_{\alpha\beta} \, dx^\alpha \otimes dx^\beta . \]
Since $\nabla^2 f = \lambda g$ we obtain by comparison:
\begin{align}
  f'' & = \lambda \label{eq:fpp-lam} \\
  \frac{1}{2} f' \de_r g_{\alpha\beta} & = \lambda g_{\alpha\beta} . \nonumber
\end{align}
In particular $\de_r \log g_{\alpha\beta} = 2 \de_r \log f'$, so
\[ g_{\alpha\beta}(x, r) = \left[ \frac{f'(r)}{f'(0)} \right]^2 g_{\alpha\beta}(x, 0) \]
(at least if $g_{\alpha\beta}(x, 0) > 0$; if it is negative we can use the
same argument for $-g_{\alpha\beta}$; if it is zero, then it remains
identically zero; in any case, the formula above holds). In other
words, $g$ is a warped metric defined over the base space $\{f=0\}$.

Let us now recall that $f$ satisfies by hypothesis the system~\eqref{eq:geom-system},
so it must hold 
$$\lambda = -f \frac{\abs{\nabla f}^2_g}{1-f^2}\,.
$$
Expanding in~\eqref{eq:fpp-lam}, we have $f'' = -f
\frac{\abs{\nabla f}_g^2}{1-f^2} = - \frac{f{(f')}^2}{1-f^2}$. This
implies that
\[ \frac{1}{2} \frac{\de}{\de r} \left[\frac{{(f')}^2}{1-f^2} \right] = \frac{f' f'' (1-f^2) + {(f')}^2 f f'}{{(1-f^2)}^2} = \frac{-{(f')}^3 f + {(f')}^3 f}{{(1-f^2)}^2} = 0 , \]
and confronting with estimate~\eqref{eq:est_f2}, we deduce ${(f')}^2 = (1-f^2)$ on the whole $\R^n\setminus\Omega$.  
It follows that the
equation for $f$ can be rewritten as
\begin{equation*}
\def\arraystretch{1.3}
\left\{ \begin{array}{r@{}c@{}l}
  f'' \,& \,=\, &\, - f ,\\
  f(0)\, & \,=\, &\, 0 ,\\
  f'(0) \,& \,=\, &\, 1 ,
  \end{array}\right.
\end{equation*}
whose solution is $f(r) = \sin (r)$. This implies that $f$
is defined and regular for $r \in \coint{0}{\frac{\pi}{2}}$, which
coincides with the region $f \in \coint{0}{1}$. 
This proves that $S=1$, and the construction above
gives an isometry of the whole space
$\R^n \setminus \Omega$ with $\coint{0}{\frac{\pi}{2}} \times \de\Omega$. The proposition
is therefore established.
\end{proof}

\noindent These results are expected in view of the thesis we are aiming to,
according to which the metric $g$ is an hemisphere and the hypersurface
$\{f=0\}$ is its equator. In order to conclude, it is enough to show
that the metric on the base space $\de\Omega$ is isometric to a round
sphere (so far we have not even proved it is homeomorphic to a
sphere).

The last key clue lies in the fact that, as seen in~\eqref{eq:def-h}
and in the following discussion, if the end of $\R^n$ is compactified
by adding the point at infinity, then the metric $g$ can be
extended at the infinity and becomes isometric to $\gamma$, which is also
$\C^1$ at infinity. The warped structure of $g$ then converts to a set
of normal coordinates for $\gamma$ at the north pole, allowing us to
trigger the following lemma taken from~\cite[Claim 4 at page
133]{caochen3}.

\begin{lem}
  Let $(M^n, \gamma)$ be a complete Riemannian manifold and $N$ a point in
  its interior, such that the metric $\gamma$ is smooth on $M \setminus \{N\}$ and
  $\C^1$ at $N$; call $r \colon M \to \R_{\geq 0}$ the distance
  function from $N$.  Suppose there exists $\epsilon > 0$, a function
  $w \colon \coint{0}{\epsilon} \to \R$ and a closed Riemannian
  manifold $(\Sigma^{n-1}, g_\Sigma)$ such that the punctured ball
  $\set{x \in M | 0 < r(x) < \epsilon}$ is isometric to the manifold
  $(\ooint{0}{\epsilon} \times \Sigma, g)$, where
  \begin{equation}
    g(r, \theta) = dr \otimes dr + w^2(r) g_\Sigma(\theta) , \qquad r \in \ooint{0}{\epsilon}, \theta \in \Sigma ,
    \label{eq:warped-metric}
  \end{equation}
  and that the isometry sends geodesic spheres on $\Sigma$ slices.  Suppose
  moreover that $w(0) = 0$ and $w'(0) \neq 0$. Then $(\Sigma, g_\Sigma)$ is
  isometric to the standard round sphere $S^{n-1}$ with its canonical
  metric, up to a constant factor.
\end{lem}

\noindent We mention the proof in~\cite{caochen3} for completeness.

\begin{proof}
  Since $\gamma$ is $\C^1$ at $N$, Christoffel coefficients are well
  defined and the regular machinery to have local existence of
  geodesics works. We can therefore consider normal coordinates $x^1$,
  \dots, $x^n$ around $N$; the metric can thus be written locally as
  \[ g(x) = \big[\delta_{ij} + \sigma_{ij}(x)\big] dx^i \otimes dx^j , \]
  with $\sigma_{ij} = o(\abs{x})$ as $\abs{x} \to 0$. This implies,
  in particular, that sufficiently small geodesic spheres are actually
  diffeomorphic to $S^{n-1}$, so $\Sigma$ is as well.

  In order to compare this with~\eqref{eq:warped-metric}, we switch to
  polar coordinates. Let $\theta^1$, \dots, $\theta^{n-1}$ be
  coordinates on $S^{n-1}$: then we can write:
  \[ x^i = r \phi^i(\theta^1, \dots, \theta^{n-1}) . \]
  Computation gives:
  \begin{align*}
    g(x) & = (\delta_{ij} + \sigma_{ij}) \cdot \left( \phi^i \phi^j \, dr \otimes dr + r \phi^j \frac{\de \phi^j}{\de \theta^\alpha} d \theta^\alpha \otimes dr + r \phi^i \frac{\de \phi^i}{\de \theta^\beta} dr \otimes d \theta^\beta + r^2 \frac{\de \phi^i}{\de \theta^\alpha} \frac{\de \phi^j}{\de \theta^\beta} d \theta^\alpha \otimes d \theta^\beta \right) \\
    & = (1 + \sigma_{ij} \phi^i \phi^j ) \cdot dr \otimes dr + r \sigma_{ij} \phi^j \frac{\de \phi^i}{\de \theta^\alpha} d \theta^\alpha \otimes dr + r \sigma_{ij} \phi^i \frac{\de \phi^j}{\de \theta^\beta} dr \otimes d \theta^\beta + \\
    & \qquad {} + \left( r^2 g_{\alpha\beta}^{S^{n-1}} + r^2 \sigma_{ij} \frac{\de \phi^i}{\de \theta^\alpha} \frac{\de \phi^j}{\de \theta^\beta} \right) d \theta^\alpha \otimes d \theta^\beta .
  \end{align*}
Comparing with~\eqref{eq:warped-metric} we have that
  \[ w^2(r) g_\Sigma(\theta) = r^2 g_{\alpha\beta}^{S^{n-1}} + r^2 \sigma_{ij}(r, \theta) \frac{\de \phi^i}{\de \theta^\alpha} \frac{\de \phi^j}{\de \theta^\beta} . \]
  Dividing by $r^2$ and passing to the limit $r \to 0$:
  \[ {\big[w'(0)\big]}^2 g_\Sigma(\theta) = g_{\alpha\beta}^{S^{n-1}} . \qedhere \]
\end{proof}

\noindent The proof of the rigidity statement is thus complete.

\appendix



\section{Well-known facts and asymptotic estimates}

\label{ap:asympt}

In this section we collect some general and widely known results about the regularity and the asymptotic behavior of solutions of system~\eqref{eq:system}. The first result justifies the assumptions made in Definition~\ref{ass:main} on the regularity of the electrostatic potential $u$.

\begin{pro}
	\label{pro:app1}
  Let $\Omega\subset\R^n$, $n \geq 3$, be a bounded open domain containing the origin, with smooth boundary and such that $\R^n\setminus\Omega$ is connected. Then there is a unique function $u$ that solves
  system~\eqref{eq:system}. Moreover, $u$ is $\C^\infty(\R^n \setminus \Omega)$, is analytic in $\R^n\setminus\overline{\Omega}$, and $u(\R^n \setminus \overline{\Omega}) = \ooint{0}{1}$.
\end{pro}

\begin{proof}
  By~\cite[Lemma 6.3.15]{helms} the infimum of the superharmonic
  functions on $\R^n \setminus \overline{\Omega}$ that take values not
  smaller than $1$ on $\de\Omega$ and not smaller than $0$ at infinity
  coincides with the supremum of subharmonic functions on the same
  domain that take values not larger than $1$ on $\de\Omega$ and not
  larger than $0$ at infinity. This common value must be a solution
  of~\eqref{eq:system} and it is at least continuous.

  By~\cite[Corollary 2.5.10]{helms}, $u$ satisfies the (strong)
  maximum and minimum principles, where the definition of the boundary
  of the domain must be extended to include the point at
  infinity. This implies that the solution is unique (because the
  difference of two solutions must be zero) and proves the given
  range.

  Higher regularity is now a completely local
  fact. Using~\cite[Theorem 6.14]{gilbarg-trudinger} we prove that
  $u \in \C^{2,\alpha}(\R^n \setminus \Omega)$. Let us now consider a
  partial derivative of $u$
  \[ v = \frac{\de u}{\de x_i} = \lim_{h\to 0} \frac{u(x+he_i)-u(x)}{h} . \]
  By linearity the difference quotient on the right is harmonic in
  $x$, so by~\cite[Theorem 2.8]{gilbarg-trudinger} also $v$ is. So we
  also have $v \in \C^{2,\alpha}(\R^n \setminus \Omega)$. The same
  procedure can be repeated for all derivatives of any order, so
  smoothness of $u$ is proved.
  
  The analyticity of $u$ can be recovered by proving that the Taylor series of $u$ converges at every point of $\R^n\setminus\overline{\Omega}$. This can be done using the mean-value property of harmonic functions in order to find appropriate bounds for the derivatives of $u$. We do not give the details, that can be found, for instance, in~\cite[Theorem~10 of Chapter~2]{evans-pde}.
\end{proof}

\noindent Now we pass to prove the asymptotic estimates widely used for $u$ and $f$. The
asymptotic estimates for $u$ are well-known, but their complete proof is rarely explictly cited: here we fill this gap.
\begin{pro}
  \label{pro:est-u}
Let $(\Omega,u)$ be a regular solution to problem~\eqref{eq:system} in the sense of Definition~\ref{ass:main}, and let $\Ca(\Omega)$ be defined by~\eqref{eq:capacity}. Then the following
  estimates hold for $\abs{x} \to \infty$:
  \begin{align*}
    u &= \Ca(\Omega)\abs{x}^{2-n} \cdot (1+o(1)) \\
    D_i u &= -(n-2)\Ca(\Omega)\abs{x}^{-n}x_i \cdot (1+o(1)) \\
    D^2_{ij}u &= (n-2)\Ca(\Omega)\abs{x}^{-2-n}\big(n x_i x_j-\abs{x}^{2}g_{ij}^{\R^n}\big) \cdot (1+o(1)) .
  \end{align*}
\end{pro}

\begin{rem}
  By the uniqueness of solutions, $\Omega$ is a round ball if and only
  if the errors in the three formul\ae\ vanish.
\end{rem}

\begin{proof}
  Suppose that $u$ is extended to assume the value $1$ inside
  $\Omega$, therefore being defined on the whole $\R^n$. The resulting
  function, called again $u$, is Lipschitz continuous and superharmonic (according to~\cite[Definition 3.3.3]{helms}), because it is harmonic on
  $\R^n \setminus \overline{\Omega}$ and all the remaining points are maximum
  points. We want to use the Riesz decomposition theorem stated in~\cite[Theorem~4.5.11]{helms}: $\R^n$ is Greenian by~\cite[Theorem~4.2.11]{helms} and its Green function is
  $\abs{x-y}^{2-n}$ (compare with~\cite[Definition~4.2.3 and the end
  of Section~2.3]{helms}). It follows that there are a measure $\mu$
  supported on a compact set $K \subset \R^n$ and an harmonic function
  $h \colon \R^n \to \R$ such that
  \[ u(x) - h(x) = U(x) := \int_{\R^n} \abs{x-y}^{2-n} \d \mu(y) \qquad \forall x \in \R^n . \]
  By tracking down the construction of the measure $\mu$, in particular looking at the proof of~\cite[Lemma~4.5.9]{helms}, we can see that, in the distributional sense,
  \begin{equation}
    \label{eq:mu}
    \mu = -\frac{1}{(n-2)\abs{S^{n-1}}} \Delta u \geq 0 .
  \end{equation}

  Since both $u$ and $U$ go to zero at infinity, $h$ does the same, so
  by harmonicity it must be identically zero. Thus:
  \begin{align}
    u(x) & = \int_{\R^n} \abs{x-y}^{2-n} \d \mu(y) \nonumber \\
         & = \abs{x}^{2-n} \cdot \mu(\R^n) + \int_{\R^n} \left( \abs{x-y}^{2-n} - \abs{x}^{2-n} \right) \d \mu(y) . \label{eq:u-repr}
  \end{align}
  From~\eqref{eq:mu} we know that $\mu$ is supported on $\de\Omega$,
  therefore $y$ ranges on a compact set. This implies that
  $\abs{x-y}^{2-n} - \abs{x}^{2-n} = o(\abs{x}^{2-n})$ for
  $\abs{x} \to \infty$. In order to prove the estimate for $u$ we then
  just need to show that $\mu(\R^n) = \Ca(\Omega)$. We can choose a
  sequence of smooth functions $u_j$ converging to $u$ in the Lipschitz norm and
  such that $u$ and $u_j$ coincide, when $j > \epsilon^{-1}$, out of the set
  \[ \Omega_\epsilon := \Set{x \in \R^n | \exists y \in \de \Omega
      \text{ s.t. } \abs{x-y} < \epsilon } . \]
  We also define the (possibly signed) measures $\mu_j$ having density
  $-\frac{1}{(n-2)\abs{S^{n-1}}} \Delta u_j$ with respect to the
  Lebesgue measure; by~\eqref{eq:mu} it holds $\mu_j \to \mu$ as
  $j \to \infty$ in the distributional sense.

  For sufficiently small
  values of $\epsilon$, the boundary of $\Omega_\epsilon$ is smooth and it
  does not intersect $\de\Omega$. Therefore we decompose it in an
  internal and an external part:
  \[ \de \Omega_\epsilon = \de^+ \Omega_\epsilon \sqcup \de^- \Omega_\epsilon := (\de \Omega_\epsilon \setminus \Omega) \sqcup (\de \Omega_\epsilon \cap \Omega) . \]
  By construction $u$, and therefore $u_j$, are harmonic out of
  $\Omega_\epsilon$. Thus, calling $\nu$ the exterior unit-normal to $\de\Omega_\epsilon$, we have
  \[
    (n-2) \abs{S^{n-1}} \cdot \mu_j(\R^n) = -\int_{\R^n} \Delta u_j \d v
    = \int_{\Omega_\epsilon} \Delta u_j \d v = \int_{\de^+ \Omega_\epsilon} \scal{D u_j}{\nu} \d \sigma + \cancel{\int_{\de^- \Omega_\epsilon} \scal{D u_j}{\nu} \d \sigma} ,
  \]
  where the cancelled term is zero because $u_j$ coincides with $u$
  (which is constant) on the integration domain.
  For $j \to +\infty$, the left hand side tends to
  $(n-2)\abs{S^{n-1}}\cdot\mu(\R^n)$, while the right hand side tends to
  $\int_{\de^+ \Omega_\epsilon} \scal{D u}{\nu} \d \sigma$. For
  $\epsilon \to 0$, this integral goes to
  \[ (n-2) \abs{S^{n-1}} \cdot \mu(\R^n) = \int_{\de\Omega} \scal{Du}{\nu} \d \sigma = \int_{\de\Omega} \abs{Du} \d \sigma = (n-2) \abs{S^{n-1}} \cdot \Ca(\Omega) , \]
  using~\eqref{eq:capacity} and the fact that $\nu=D u/\abs{D u}$ on $\de\Omega$. The result for $u$ is thus established.

  Estimates for $Du$ and $D^2 u$ follow using the same pattern,
  differentiating under the integral sign in~\eqref{eq:u-repr}. A
  proof for exchanging derivation and integration over a general
  measure is~\cite[Lemma 2.5.5]{helms}, whose hypotheses are easy to
  verify.
\end{proof}

\noindent As a first consequence of Proposition~\ref{pro:est-u}, we have that, for any $\varepsilon>0$ sufficiently small, the set $\set{u=\varepsilon}$ is regular and diffeomorphic to a sphere. Moreover, one easily computes
$$
\frac{D_i u}{\abs{D u}}\,=\,-\frac{x_i}{\abs{x}}\cdot \big(1+o(1)\big)\,.
$$
Therefore, the level sets of $u$ tend to coincide with the level sets of $|x|$ as $|x|\to +\infty$. In particular, for small $\epsilon$'s, the level set $\{u=\epsilon\}$ can be radially projected to 
$$
S_\epsilon=\Big\{\abs{x}=\left[\frac{\Ca(\Omega)}{\epsilon}\right]^{\frac{1}{n-2}}\Big\}\,,
$$ 
and the push-forward of the measure $\sigma$ of $\{u=\epsilon\}$ on $S_\epsilon$ tends to coincide with the spherical measure as $\epsilon$ goes to $0$. From this we deduce
\begin{equation}
\label{eq:est-levels}
\int_{\{u=\varepsilon\}}\!\!d\sigma= \left[\frac{\Ca(\Omega)}{\epsilon}\right]^{\frac{n-1}{n-2}}\abs{S^{n-1}}+o(\epsilon^{-\frac{n-1}{n-2}})
\end{equation}
as $\epsilon\to 0^+$.


Finally, we discuss some alternative definitions for the capacity of a set $\Omega$ and we prove their equivalence.

\begin{pro}
	Let $(\Omega,u)$ be a regular solution to problem~\eqref{eq:system} in the sense of Definition~\ref{ass:main}, and let $\Ca(\Omega)$ be the capacity of $\Omega$, defined by~\eqref{eq:def-omega}. Then the following two formul\ae\ hold
	\begin{align}
	\label{eq:capacity_alter_1}
	\Ca(\Omega) \,&=\,  \frac{1}{(n-2) \abs{S^{n-1}}} \int_{\de\Omega} \abs{Du} \d \sigma \, ,
	\\
	\label{eq:capacity_alter_2}
	\Ca(\Omega) \,&=\,\frac{1}{|S^{n-1}|}\int_{\de\Omega} \abs*{\frac{Du}{n-2}}^2\scal{x}{\nu} \d \sigma\,,
	\end{align}
where $x$ is the position vector in $\R^n$.
\end{pro}

\begin{proof}
Let $\Omega$ be our bounded domain, and consider the functional
$$
\mathcal{F}: \left\{w\in\mathcal{C}_c^{\infty}(\R^n),\ w\equiv 1\text{ in }\Omega\right\}\longrightarrow \R\,,
\qquad\mathcal{F}: w\,\longmapsto\,\frac{1}{(n-2)|S^{n-1}|}\int_{\R^n}\abs{D w}^2\d v\,.
$$
First of all, we show that the functional $\mathcal{F}$ is strictly convex. An easy computation shows that, for every $w_1,w_2\in\mathcal{C}_c^{\infty}(\R^n)$ with $w_1\equiv w_2\equiv 1$ on $\Omega$, 
\begin{align}
\notag
\mathcal{F}\left(\frac{w_1+w_2}{2}\right)
\,&=\,
\frac{1}{4}\mathcal{F}(w_1)+\frac{1}{4}\mathcal{F}(w_2)+\frac{1}{2(n-2)|S^{n-1}|}\int_{\R^n}\scal{Dw_1}{Dw_2}\d v
\\
\notag
\,&=\,\frac{\mathcal{F}(w_1)+\mathcal{F}(w_2)}{2}-\frac{1}{4(n-2)|S^{n-1}|}\int_{\R^n}\left(\abs{Dw_1}^2+\abs{Dw_2}^2-2\scal{Dw_1}{Dw_2}\right)\d v
\\
\label{eq:F_conv}
&\leq\,\frac{\mathcal{F}(w_1)+\mathcal{F}(w_2)}{2}\,,
\end{align}
where in the last step we have used the Cauchy-Schwartz inequality. Moreover, if the equality holds in~\eqref{eq:F_conv}, then $w_1=w_2+c$, with $c$ constant. But $w_1=w_2$ on $\Omega$, hence $w_1=w_2$ everywhere. This proves that $\mathcal{F}$ is strictly convex.

Now let $u$ be the solution of system~\eqref{eq:system}. We want to prove that $u$ is the minimum of $\mathcal{F}$. To this end, since $\mathcal{F}$ is strictly convex, it is enough to prove that $u$ is a critical point of $\mathcal{F}$, that is, we want to show
\begin{equation*}
\lim_{\epsilon\to 0^+}\frac{\mathcal{F}(u+\epsilon w)-\mathcal{F}(u)}{\epsilon}\,=\,0\,,
\end{equation*}
for every $w\in\mathcal{C}_c^{\infty}(\R^n\setminus\overline{\Omega})$. This is an easy computation
\begin{align}
\notag
\lim_{\epsilon\to 0^+}\frac{\mathcal{F}(u+\epsilon w)-\mathcal{F}(u)}{\epsilon}
\,&=\,
\lim_{\epsilon\to 0^+}\frac{1}{(n-2)|S^{n-1}|}\int_{\R^n}\left(2\scal{D w}{D u}+\epsilon\abs{Dw}^2\right)\d v
\\
\notag
&=\,\frac{2}{(n-2)|S^{n-1}|}\int_{\R^n\setminus\Omega}\scal{D w}{D u}\d v
\\
\label{eq:ucrit}
&=\,\frac{2}{(n-2)|S^{n-1}|}\left(-\int_{\R^n\setminus\Omega}w\Delta u\d v+\int_{\de\Omega}u\scal{Dw}{\nu}\right)\,,
\end{align}
where $\nu=-Du/\abs{Du}$ is the unit normal to $\de\Omega$ and in the latter equality we have integrated by parts. Since $u$ is harmonic and $w\in\C_c^{\infty}(\R^n\setminus\overline{\Omega})$, the limit in~\eqref{eq:ucrit} is zero. Therefore, the function $u$ is indeed the minimum of $\mathcal{F}$, and, from~\eqref{eq:def-omega} we obtain 
\begin{equation}
\label{eq:capacity_alter_3}
\Ca(\Omega)\,=\,\frac{1}{(n-2)|S^{n-1}|}\int_{\Omega}\abs{D u}^2\d\sigma\,.
\end{equation}
Identity~\eqref{eq:capacity_alter_1} now follows immediately integrating by parts and using the harmonicity of $u$.

In order to prove~\eqref{eq:capacity_alter_2}, we consider the tensor
$$
T\,=\,du\otimes du-\frac{|D u|^2}{2}g_{\R^n}\,.
$$
An easy computation shows that $T$ has zero divergence
$$
\Div T\,=\,\cancel{\Delta u Du}+\frac{1}{2}D\abs{Du}^2-\frac{1}{2}D\abs{Du}\,=\,0\,.
$$
Integrating by parts, for all vector fields $X$ on $\R^n\setminus\Omega$ we obtain
\begin{equation}
\label{eq:poho}
0\,=\,\int_{\R^n\setminus\Omega}\scal{\Div T}{X}\d v\,=\,-\int_{\de\Omega}T(X,\nu)\d\sigma-\int_{\R^n\setminus\Omega}\scal{T}{DX}\d v\,,
\end{equation}
where $\nu=-Du/|D u|$, as usual. If we choose $X$ as the position vector, that is $X_{|_{x}}=x$, we obtain
\begin{align*}
\scal{T}{Dx}\,&=\,T_{ij}\delta^{ij}\,=\,
\abs{Du}^2-\frac{\abs{Du}^2}{2}\cdot n
\,=\,
-\frac{n-2}{2}\,\abs{Du}^2\,,
\\
T(x,\nu)\,&=\,-T\left(x,\frac{Du}{\abs{Du}}\right)\,=\,-\scal{Du}{x}\cdot\Scal{Du}{\frac{Du}{\abs{Du}}}+\frac{\abs{Du}^2}{2}\Scal{x}{\frac{Du}{\abs{Du}}}\,=\,\frac{\abs{Du}^2}{2}\scal{x}{\nu}\,.
\end{align*}
Substituting in~\eqref{eq:poho}, we find
$$
(n-2)\int_{\R^n\setminus\Omega}\abs{Du}^2\d v\,=\,\int_{\de\Omega}\abs{Du}^2\scal{x}{\nu}\d\sigma\,,
$$
and using~\eqref{eq:capacity_alter_3}, we obtain~\eqref{eq:capacity_alter_2}.
\end{proof}

\noindent Now that the different definitions of capacity have been discussed, we pass to the proof of a well known isoperimetric inequality, referred to as the Poincar\'{e}-Faber-Szeg\"o inequality. This inequality has been proved in~\cite[Section~1.12]{Pol_Sze} for the $3$-dimensional case, while the proof in the $n$-dimensional case has been summarized in~\cite{Jauregui}.

\begin{theo}
\label{theo:PFS}
Let $\Omega\subset \R^n$ bounded domain with smooth boundary $\partial\Omega$. Then it holds
\begin{equation}
\label{eq:PFS_app}
\left(\frac{|\Omega|}{|B^n|}\right)^{\!\frac{n-2}{n}}\!\!\leq\,{\rm Cap}(\Omega)\,.
\end{equation}
Moreover, if the equality holds, then $\Omega$ is a ball.
\end{theo}

\begin{proof}
Let $u$ be the solution of problem~\eqref{eq:system} corresponding to the domain $\Omega$, and, for all $t\in(0,1]$, set $\Sigma_t=\{u=t\}\subset \R^n\setminus\Omega$. We have already remarked that $u$ is analytic, hence, by the results in~\cite{Sou_Sou}, we have that the critical values of $u$ are discrete. In particular, for almost all $t\in(0,1]$, the set $\Sigma_t$ is a smooth hypersurface. Moreover, using the results in~\cite{Har_Sim,Lin} we have that the $(n-1)$-dimensional Hausdorff measure of $\Sigma_t$ is finite for all $t$. We also observe that, using the coarea formula, the capacity of $\Omega$ can be written as
\begin{equation}
\label{eq:capacity_coarea}
{\rm Cap}(\Omega)\,=\,\frac{1}{(n-2)|S^{n-1}|}\int_{\R^n\setminus\Omega}|D u|^2\d v\,=\,\frac{1}{(n-2)|S^{n-1}|}\int_0^1\left(\int_{\Sigma_t}|D u|\d\sigma_t\right)\d t\,,
\end{equation}
where $\sigma_t$ is the measure induced on $\Sigma_t$.
Fix now a regular value $t\in(0,1]$, so that $|D u|\neq 0$ on $\Sigma_t$.
Our aim is to prove that $\int_{\Sigma_t}\abs{D u}\d\sigma_t\geq\int_{S_t}\abs{D w}\d s_t$, where $w$ is the symmetrization of $u$ (see the definition below) and in fact its level sets $S_t$'s are geodesic spheres bounding geodesic balls having the same volumes as the regions enclosed by the $\Sigma_t$'s.  
Using Cauchy--Schwartz, we compute
\begin{align*}
|\Sigma_t|^2\,&=\,\left(\int_{\Sigma_t}\d\sigma_t\right)^2
\\
&=\,\left(\int_{\Sigma_t}\sqrt{|D u|}\cdot\frac{1}{\sqrt{|D u|}}\d\sigma_t\right)^2
\\
&\leq\,\left[\left(\int_{\Sigma_t}|D u|\d\sigma_t\right)^{\!\frac{1}{2}}\left(\int_{\Sigma_t}\frac{1}{|D u|}\d\sigma_t\right)^{\!\frac{1}{2}}\right]^2
\\
&=\,\left(\int_{\Sigma_t}|D u|\d\sigma_t\right)\left(\int_{\Sigma_t}\frac{1}{|D u|}\d\sigma_t\right)\,,
\end{align*}
for every $t$ regular value of $u$.
Therefore, for regular values, the following estimate holds for the integrand in formula~\eqref{eq:capacity_coarea}
\begin{equation}
\label{eq:capacity_ineq1}
\int_{\Sigma_t}|D u|\d\sigma_t\,\geq\,\frac{|\Sigma_t|^2}{\int_{\Sigma_t}\frac{1}{|D u|}\d\sigma_t}\,.
\end{equation}
Now, let us call $\Omega_t=\Omega\cup\{t\leq u\leq 1\}$. In particular, we have $\partial\Omega_t=\Sigma_t$ and
\begin{equation}
\label{eq:area_Omegat}
|\Omega_t|\,=\,|\Omega|+\int_t^1\left(\int_{\Sigma_s}\frac{1}{|D u|}\d\sigma_s\right)\d s\,.
\end{equation}
Consider the function $t\mapsto|\Omega_t|$, defined from $(0,1]$ to $[|\Omega|,+\infty)$. We prove now that $t\mapsto|\Omega_t|$ is continuous for all $t$. 
To do that, it is sufficient to show the continuity of the repartition function 
$$
F(t)\,=\,|\{u>t\}|\,=\,\int_{\{u>t\}}\d v\,=\,\int_t^1\left(\int_{\Sigma_\tau}\frac{1}{|D u|}\d\sigma_\tau\right)\d \tau\,.
$$
Since the measure on $\R^n\setminus\Omega$ is locally finite and positive, from~\cite[Proposition~2.6]{Amb_Dap_Men} it follows that, for any value $t_0\in(0,1]$, it holds
$$
\lim_{t\to t_0^+}F(t)\,=\,F(t_0)\,,\qquad \lim_{t\to t_0^-}F(t)\,=\,F(t_0)+|\{u=t_0\}|\,.
$$ 
On the other hand, $|\{u=t_0\}|=|\Sigma_{t_0}|=0$, because we have already observed that the Hausdorff dimension of $\Sigma_t$ is $(n-1)$. Therefore $F$ is continuous and so is $t\mapsto |\Omega_t|$.

Moreover, the function $t\mapsto |\Omega_t|$ is clearly monotonic and strictly decreasing in $t$, and it is differentiable for all regular values $t$ of $u$ (hence, it is differentiable almost everywhere, because the critical values of $u$ are discrete). At a point where $t\mapsto|\Omega_t|$ is differentiable, from~\eqref{eq:area_Omegat} and the Fundamental Theorem of Calculus, we deduce
\begin{equation}
\label{eq:der_Omegat}
\frac{d|\Omega_t|}{dt}(t)=-\int_{\Sigma_t}\frac{1}{|D u|}\d\sigma_t\,.
\end{equation}
Moreover, from the isoperimetric inequality we have  
\begin{equation}
\label{eq:iso_in}
\left(\frac{|\Omega_t|}{|B^n|}\right)^{\frac{1}{n}}\,\leq\,\left(\frac{|\Sigma_t|}{|S^{n-1}|}\right)^{\frac{1}{n-1}}\,.
\end{equation}
Using formul\ae~\eqref{eq:der_Omegat} and~\eqref{eq:iso_in} to estimate the right hand side of~\eqref{eq:capacity_ineq1} we obtain
\begin{equation}
\label{eq:capacity_ineq2}
\int_{\Sigma_t}|D u|\d\sigma_t\,\geq\,-|S^{n-1}|^2\left(\frac{|\Omega_t|}{|B^n|}\right)^{2\frac{n-1}{n}}{\frac{1}{\frac{d|\Omega_t|}{dt}}}\,.
\end{equation}
We define now the function $R:(0,1]\to[R(1),+\infty)$ as
$$
R(t)\,=\,\left(\frac{|\Omega_t|}{|B^n|}\right)^{\frac{1}{n}}\,.
$$
Since $t\mapsto |\Omega_t|$ is continuous, so is $t\mapsto R(t)$. Moreover, $R$ is strictly decreasing, differentiable when $t\mapsto|\Omega_t|$ is differentiable, and in that case it holds
$$
\frac{d|\Omega_t|}{dt}(t)\,=\,n|B^n|R^{n-1}(t)R'(t)\,=\,|S^{n-1}|R^{n-1}(t)R'(t)\,,
$$
where we have denoted by $R'(t)$ the derivative of $R(t)$ with respect to $t$ (where it exists).
Formula~\eqref{eq:capacity_ineq2} can be rewritten in terms of $R$ as
\begin{equation}
\label{eq:capacity_ineq3}
\int_{\Sigma_t}|D u|\d\sigma_t\,\geq\,-|S^{n-1}|\frac{R^{n-1}(t)}{R'(t)}\,.
\end{equation}
Let $(r,\theta)\in\R_+\times S^{n-1}$ be the polar coordinates on $\R^n$ and let $B(0,r)$ be the ball of radius $r$ centered at the origin. 
Define $w:\R^n\setminus B(0,R(1))\to (0,1]$ as
$$
w(r,\theta)=R^{-1}(r)\,,\ \forall\theta\in S^{n-1}\,,
$$
where $R^{-1}$ is the inverse of $R$.
When $t$ is a regular value, we have
$$
\frac{\partial w}{\partial r}(R(t),\theta)\,=\,\frac{1}{R'(t)}\,,\qquad \frac{\partial w}{\partial \theta}=0\,,
$$
so that
$$
|D w|(R(t),\theta)\,=\,\abs*{\frac{1}{R'(t)}}\,=\,-\frac{1}{R'(t)}\,.
$$
In particular, $|D w|$ is independent of $\theta$, and integrating it on the sphere of radius $R(t)$ centered at the origin, that is $S_t=\de B(0,R(t))$, we obtain
$$
\int_{S_t}|D w|\d s_t\,=\,-|S_t|\,\frac{1}{R'(t)}\,=\,-|S^{n-1}|\frac{R^{n-1}(t)}{R'(t)}\,,
$$
where we have denoted by $s_t$ the measure induced on $S_t$.
Therefore, from formula~\eqref{eq:capacity_ineq3} we deduce the following inequality
\begin{equation}
\label{eq:capacity_ineq4}
\int_{\Sigma_t}|D u|\d\sigma_t\,\geq\,\int_{S_t}|D w|\d s_t\,.
\end{equation}
We remark that formula~\eqref{eq:capacity_ineq4} holds only for regular values $t$ of $u$. However, we have already observed that the critical values of $u$ are discrete, hence formula~\eqref{eq:capacity_ineq4} holds for almost every $t\in(0,1]$. Integrating in $t$ and recalling inequality~\eqref{eq:capacity_coarea}, we obtain
\begin{equation*}
{\rm Cap}(\Omega)\,\geq\,
\frac{1}{(n-2)|S^{n-1}|}\int_0^1\left(\int_{S_t}|D w|\d s_t\right)\d t
=\,\frac{1}{(n-2)|S^{n-1}|}\int_{\R^n\setminus B(0,R(1))}\!\!|D w|^2\d\mu\,,
\end{equation*}
where in the rightmost equality we have used the coarea formula. From the definition of capacity, it follows that
\begin{equation}
\label{eq:capacity_ineq5}
{\rm Cap}(\Omega)\,\geq\,{\rm Cap}(B(0,R(1)))\,.
\end{equation}
The capacity of a ball is easily computed, and recalling the definition of $R$ we find
$$
{\rm Cap}(B(0,R(1)))\,=\,(R(1))^{n-2}\,=\,\left(\frac{|\Omega|}{|B^n|}\right)^{\!\frac{n-2}{n}}\!\!.
$$ 
Substituting in~\eqref{eq:capacity_ineq5}, we obtain inequality~\eqref{eq:PFS_app}. If the equality holds in~\eqref{eq:PFS_app}, retracing the steps of this proof, one sees that the equality also holds in the isoperimetric inequality~\eqref{eq:iso_in} for all $t$, so that the $\Omega_t$'s (in particular $\Omega=\Omega_1$) are forced to be balls.
\end{proof}

\noindent We also remark that, if the domain $\Omega$ satisfies
\begin{equation}
\label{eq:isoin_capsurf}
\left(\frac{|\de\Omega|}{|S^{n-1}|}\right)^{\!\frac{n-2}{n-1}}\!\!\leq\,\Ca(\Omega) \,,
\end{equation}
then the isoperimetric inequality would immediately imply~\eqref{eq:PFS_app}. However, inequality~\eqref{eq:isoin_capsurf} is known to be false in general. In fact, one can enclose a domain with arbitrarily large boundary measure $|\de\Omega|$ inside a ball $B(0,r)$ of small fixed radius $r$, so that $\Ca(\Omega)\leq\Ca(B(0,r))=r^{n-2}<<|\de\Omega|$. This means that the quantity
$$
\mathcal{E}_n(\Omega)\,=\,\frac{\Ca(\Omega) }{\left(\frac{|\de\Omega|}{|S^{n-1}|}\right)^{\frac{n-2}{n-1}}}
$$
can be made arbitrarily small.
One may then ask if a lower bound for $\mathcal{E}_n(\Omega)$ exists at least for convex domains. In dimension $n=3$, the limit of the quotient
$$
\mathcal{E}_3(\Omega)\,=\,\frac{\Ca(\Omega)}{\sqrt{\frac{|\de\Omega|}{4\pi}}}\,,
$$
as $\Omega$ approaches the $2$-disk in $\R^3$, is $2\sqrt{2}/\pi$, and a well known conjecture by P\'{o}lya-Szeg\"o~\cite[Section~1.18]{Pol_Sze} states that the value $2\sqrt{2}/\pi$ is indeed the infimum of $\mathcal{E}_3(\Omega)$ among $3$-dimensional convex domains with positive surface area, see also~\cite{Cra_Fra_Gaz,Fra_Gaz_Pie} for some discussions on the topic.

 We end this section with two corollaries on the asymptotic behavior of $f$ and $g$, whose proof follow by just applying Proposition~\ref{pro:est-u} together with the formul\ae\ derived
in Sections~\ref{sec:conformal-form} and~\ref{sec:estimates}.

\begin{cor}
  \label{cor:est-f}
  Let $(\Omega,f,g)$ be a regular solution to problem~\eqref{eq:geom-system} in the sense of Definition~\ref{ass:fg}. Then the following asymptotic estimates
  hold for $\abs{x} \to \infty$:
  \begin{align*}
    f &= 1-2\Ca(\Omega)^{\frac{2}{n-2}}\abs{x}^{-2} \cdot (1+o(1)) \\
    \nabla_i f &= 4 \cdot \Ca(\Omega)^{\frac{2}{n-2}}\abs{x}^{-4}x_i \cdot (1+o(1)) \\
    \nabla^2_{ij}f &= -4 \cdot \Ca(\Omega)^{\frac{2}{n-2}}\abs{x}^{-4}g_{ij}^{\R^n} \cdot (1+o(1)) .
  \end{align*}
\end{cor}

\begin{cor}
  \label{cor:asympt}
  Let $(\Omega,f,g)$ be a regular solution to problem~\eqref{eq:geom-system} in the sense of Definition~\ref{ass:fg}. Then the following asymptotic
  estimates hold for $\abs{x} \to \infty$:
  \begin{align*}
    \abs{\nabla f}_g & = 2 \abs{x}^{-1} \cdot (1+o(1)) \\
    H_g & = -\frac{n-1}{2} {\Ca(\Omega)}^{-\frac{2}{n-2}} \abs{x} \cdot (1+o(1)) \\
    d v_g & = 2^n \cdot {\Ca(\Omega)}^{-2\frac{n}{n-2}} \abs{x}^{-2n} \cdot (1+o(1)) \cdot \d v \\
    \nu_g^i & = \left[2 \cdot {\Ca(\Omega)}^{\frac{2}{n-2}} \right]^{-1} \abs{x} x_j g^{ij}_{\R^n} \cdot (1+o(1)) \\
    g_{ij} & = 4 \cdot {\Ca(\Omega)}^{\frac{4}{n-2}} \abs{x}^{-4} \cdot (1+o(1)) \\
    X^i & = \abs{x}^2 x_i \cdot o(1) .
  \end{align*}
\end{cor}

\noindent
We also mention that from formula~\eqref{eq:est-levels} we obtain the following estimate on the area of the level sets of the function $f$:
\begin{equation}
\label{eq:est-levels-f}
\int_{\{f=1-\varepsilon\}}\!\!d\sigma_g<C\,\epsilon^{\frac{n-1}{2}}
\end{equation}
for some constant $C>0$ and for $0<\epsilon<1$ big enough. Of course one can derive a more precise estimate, but~\eqref{eq:est-levels-f} is enough for the purposes of this work.

\section*{Acknowledgements.} The authors are members of Gruppo Nazionale per l'Analisi Matematica, la Probabilit\`a e le loro Applicazioni (GNAMPA), which is part of the Istituto Nazionale di Alta Matematica (INdAM), and partially funded by the GNAMPA project ``Principi di fattorizzazione, formule di monotonia e disuguaglianze geometriche''.
The authors wish to thank Giulio Ciraolo for bringing to their attention the references~\cite{Bia_Cir_2,Bia_Cir_Sal,Bandle,Pay_Phi}.
\bibliographystyle{amsalpha}
\bibliography{bibliography}

\end{document}